\documentclass[a4paper]{amsart}

\usepackage{amsmath,amsthm,amssymb,enumerate}
\usepackage{epsfig}
\usepackage{amssymb}
\usepackage{mathtools}
\usepackage{amsmath}
\usepackage{amssymb}
\usepackage{amsmath,amsthm}
\usepackage[latin1]{inputenc}
\usepackage[T1]{fontenc}
\usepackage{path}
\usepackage{ae,aecompl}
\usepackage{amsfonts}
\usepackage{amsxtra}
\usepackage{euscript,mathrsfs}
\usepackage{color}
\usepackage[left=3.4cm,right=3.4cm,top=4cm,bottom=4cm]{geometry}
\usepackage[citecolor=blue,colorlinks=true]{hyperref}
\allowdisplaybreaks

\usepackage{enumitem}
\setenumerate{label={\rm (\alph{*})}}

\usepackage{amsfonts}
\usepackage{amsxtra}

\usepackage[frak,hyperref,theorem_section]{paper_diening}
\numberwithin{equation}{section}

\newcommand{\Div}{\divergence}

\newcommand{\R}{\mathbb R}
\newcommand{\N}{\mathbb N}

\newcommand{\E}{\mathbb E}
\newcommand{\p}{\mathbb P}

\newcommand{\F}{\mathfrak F}

\newcommand{\A}{\mathcal A}

\newcommand{\dd}{\mathrm d}
\newcommand{\dx}{\, \mathrm{d}x}
\newcommand{\dy}{\, \mathrm{d}y}

\newcommand{\ds}{\, \mathrm{d}\sigma}
\newcommand{\dt}{\, \mathrm{d}t}
\newcommand{\dxt}{\,\mathrm{d}x\, \mathrm{d}t}
\newcommand{\dxs}{\,\mathrm{d}x\, \mathrm{d}\sigma}
\newcommand{\dif}{\mathrm{d}}

\newcommand{\mf}{\mathfrak{F}}

\newcommand{\prst}{\mathbb{P}}

\newcommand{\mn}{\mathbb{N}}

\newcommand{\mt}{\mathcal{O}}
\newcommand{\tor}{\mathcal{O}}

\DeclareMathOperator{\diver}{div}

\begin{document}

\title[2D stochastic Navier--Stokes equations in bounded domains]
{Error analysis for 2D stochastic Navier--Stokes equations in bounded domains with Dirichlet data}

\author{Dominic Breit}
\address{Department of Mathematics, Heriot-Watt University, Riccarton Edinburgh EH14 4AS, UK}
\email{d.breit@hw.ac.uk}

\author{Andreas Prohl}
\address{Mathematisches Institut,
Universit\"at T\"ubingen,
Auf der Morgenstelle 10,
D-72076 T\"ubingen
Germany}
\email{prohl@na.uni-tuebingen.de}

%
%

\begin{abstract}
We study a finite-element based space-time discretisation for the 2D
stochastic Navier--Stokes equations in a bounded domain supplemented with no-slip boundary conditions. We prove optimal convergence rates in the energy norm with respect to convergence in probability, that is convergence of order (almost) 1/2 in time and 1 in space. This was previously only known in the space-periodic case, where higher order energy estimates for any given (deterministic) time are available. In contrast to this, estimates in the Dirichlet-case are only known for a (possibly large) stopping time. We overcome this problem by 
introducing an approach based on discrete stopping times. This replaces the localised estimates (with respect to the sample space) from earlier contributions.

\end{abstract}


\date{\today}

\maketitle

%
%
%
%
%
%
%
%
%
%

\section{Introduction}

We are concerned with the numerical approximation of the 2D stochastic Navier--Stokes equations in a smooth bounded domain $\mt\subset\R^2$ supplemented with no-slip boundary conditions.
They describe the flow of a homogeneous incompressible fluid in terms of the velocity field $\bfu$ and pressure function $p$ defined on a filtered probability space $(\Omega,\mathfrak F,(\mathfrak F_t),\p)$ and read as
\begin{align}\label{eq:SNS}
\left\{\begin{array}{rc}
\dd\bfu=\mu\Delta\bfu\dt-(\bfu\cdot\nabla)\bfu\dt-\nabla p\dt+\Phi(\bfu)\dd W
& \mbox{in $\mathcal O_T$,}\\
\Div \bfu=0\qquad\qquad\qquad\qquad\qquad\,\,\,\,& \mbox{in $\mathcal O_T$,}\\
\bfu(0)=\bfu_0\,\qquad\qquad\qquad\qquad\qquad&\mbox{ \,in $\mt$,}\end{array}\right.
\end{align}
$\p$-a.s. in $\mathcal O_T:=(0,T)\times\mt$, where $T>0$, $\mu>0$ is the viscosity and $\bfu_0$ is a given initial datum. The momentum equation is driven by a cylindrical Wiener process $W$ and the diffusion coefficient $\Phi$ takes values in the space of Hilbert-Schmidt operators; see Section \ref{sec:prob} for details.

Existence, regularity and long-time behaviour of solutions to \eqref{eq:SNS} have been studied extensively over the last three decades, and we refer to \cite{KukShi} for a complete picture. Most of the available results consider \eqref{eq:SNS} with respect to periodic boundary conditions. In some cases this is only for a simplification of the presentation. For instance, the existence of stochastically strong solutions to \eqref{eq:SNS}
is not effected by the boundary condition. Looking at the spatial regularity of solutions the situation is completely different:
\begin{itemize}
\item If $\mt=\mathbb T^2$ --- the two-dimensional torus --- and \eqref{eq:SNS} is supplemented with periodic boundary conditions one can obtain estimates in any Sobolev space provided the data (initial datum and diffusion coefficient) are sufficiently regular; cf. \cite[Corollary 2.4.13]{KukShi}.
\item If, on the other hand, $\mt$ is a bounded domain with smooth boundary and \eqref{eq:SNS} is supplemented with the no-slip boundary condition
\begin{align}\label{eq:nosilp}
\bfu=0\quad \mbox{on $(0,T)\times\partial\mt$,}
\end{align}
it is still an open problem if the solution satisfies
\begin{align}\label{eq:reglack}
\E\big[\|\nabla\bfu(T)\|_{L^2_x}^2\big]<\infty
\end{align}
for any given $T<\infty$, cf. \cite{GlZi,KuVi}. Regularity estimates are only kown until a (possibly large) stopping time and even with this restriction the spatial regularity seems limited; see Lemma \ref{lem:reg} (c) and Remark \ref{rem:reg}.
\end{itemize} 
Moment estimates such as \eqref{eq:reglack} are crucial
for the numerical analysis. If they are not at disposal it is unclear how to obtain convergence rates for a discretisation of \eqref{eq:SNS}.
Consequently most, if not all available results are concerned with the space-periodic problem. In particular, it is shown in \cite{BrDo} and \cite{CP} for the space-periodic problem
that for any $\xi>0$
\begin{align}\label{eq:perror}
&\mathbb P\bigg[\max_{1\leq m\leq M}\|\bfu(t_m)-\bfu_{h,m}\|_{L^2_x}^2+\sum_{m=1}^M \tau\|\nabla\bfu(t_m)-\nabla\bfu_{h,m}\|_{L^2_x}^2>\xi\,\big(h^{2\beta}+\tau^{2\alpha}\big)\bigg]\rightarrow0
\end{align}
as $h,\tau\rightarrow0$ (where $\alpha<\frac{1}{2}$ and $\beta<1$ are arbitrary); see also \cite{BeMi1,BeMi2} for related results. Here $\bfu$ is the solution to \eqref{eq:SNS} and $\bfu_{h,m}$ the approximation of $\bfu(t_m)$ with discretisation parameters $\tau=T/M$ (time) and $h$ (space).
The relation \eqref{eq:perror} tells us that the convergence in probability is of order (almost) 1/2 in time and 1 in space. It seems to be an intrinsic feature of SPDEs with general non-Lipschitz nonlinearities such as \eqref{eq:SNS}
that the more common concept of a pathwise error (an error measured in $L^2(\Omega)$) is too strong (see \cite{Pr} for first contributions). Hence \eqref{eq:perror} is the best result we can hope for. The proof of \eqref{eq:perror} is based on 
estimates in $L^2(\Omega)$, which are localised with respect to the sample set. The size of the neglected sets shrinks asymptotically with respect to the discretisation parameters and is consequently not seen
in \eqref{eq:perror}. The localised $L^2(\Omega)$-estimates in question rely on an iterative argument in the $m$-th step of which one can only control the discrete solution up to the step $m-1$ (to avoid problems with $(\mathfrak F_t)$-adaptedness), while the continuous solution is estimated by means of the global regularity estimates being available in the periodic setting (recall the discussion above). \\
In contrast to the periodic situation, in the Dirichlet-case estimates are only known for a (possibly large) stopping time since the equality $\int_{{\mathcal O}} (\bfu\cdot\nabla)\bfu\cdot\Delta\bfu\dx=0$ is no longer available. Incorporating the latter case into the framework of the localised estimates, the iterative argument just mentioned fails: controlling the continuous solution in the $m$-th step only until the time $t_{m-1}$ is insufficient for the estimates, while ``looking into'' the interval $[t_{m-1},t_m]$ in this set-up destroys the martingale character of certain stochastic integrals we have to estimate. We overcome this problem by using an approach based on discrete stopping times, which replaces the localised $L^2(\Omega)$-estimates from earlier contributions. This allows to control all quantities even in the interval $[t_{m-1},t_m]$ and, at the same time, preserves the martingale property of the stochastic integrals (see also the discussion in Remark \ref{rem}).
As a result we obtain `global-in-$\Omega$' estimates up to the discrete stopping time; cf. Theorem \ref{thm:4}. The discrete stopping times are constructed such that they converge to $T$, where $T$ can be any given end-time. Consequently, the convergence in probability as in \eqref{eq:perror} follows for the Dirichlet-case, see our main result in Theorem \ref{thm:main}. We believe that this strategy will be of use also for other SPDEs with non-Lipschitz nonlinearities.

We work under the structural assumption of a solenoidal diffusion coefficient which vanishes at the boundary. This is crucial in the regularity estimate from Lemma \ref{lem:reg} (b) in order to control the correction term $V^N(t)$ in the proof.
Due to the counterexamples concerning the regularity for stochastic PDEs in bounded domains, see \cite{Kr}, this seems to be unavoidable. In fact, the same assumptions are made in the analytical paper \cite{GlZi} on which we built on.


\section{Mathematical framework}
\label{sec:framework}

\subsection{Probability setup}\label{sec:prob}

Let $(\Omega,\F,(\F_t)_{t\geq0},\prst)$ be a stochastic basis with a complete, right-continuous filtration. The process $W$ is a cylindrical $\mathfrak U$-valued Wiener process, that is, $W(t)=\sum_{j\geq1}\beta_j(t) e_j$ with $(\beta_j)_{j\geq1}$ being mutually independent real-valued standard Wiener processes relative to $(\F_t)_{t\geq0}$, and $(e_j)_{j\geq1}$ a complete orthonormal system in a separable Hilbert space $\mathfrak{U}$.
Let us now give the precise definition of the diffusion coefficient $\varPhi$ taking values in the set of Hilbert-Schmidt operators $L_2(\mathfrak U;\mathbb H)$, where
$\mathbb H$ can take the role of various Hilbert spaces. We define $L^2_{\Div}(\mt,\R^2)$ and $W^{1,2}_{0,\Div}(\mt,\R^2)$ to be the closure of $C^\infty_{c,\Div}(\mt,\R^2)$ -- the solenoidal $C^{\infty}_c(\mt,\R^2)$-functions -- in $L^2(\mt,\R^2)$ and $W^{1,2}_{0}(\mt,\R^2)$, respectively.
We also work with fractional Sobolev spaces
$W^{\sigma,p}(0,T;X)$ for $p\in(1,\infty)$ and $\sigma\in(0,1)$ and a Banach space $(X;\|\cdot\|_X)$ with norm given by
\begin{align*}
\|f\|_{W^{\sigma,p}(0,T;X)}^p:=\|f\|_{L^p(0,T;X)}^p+\int_{0}^T\int_{0}^T\frac{\|f(t)-f(s)\|_X^p}{|t-s|^{1+\sigma p}}\,\dif s\dt.
\end{align*}
Similarly,
$W^{\sigma,p}(\mt,\R^2)$ is the fractional Sobolev space with norm given by
\begin{align*}
\|v\|_{W^{\sigma,p}_x}^p:=\|v\|_{L^p_x}^p+\int_{\mt}\int_{\mt}\frac{|v(x)-v(y)|^p}{|x-y|^{2+\sigma p}}\dx\dy.
\end{align*}
We assume that  $\Phi(\bfu)\in L_2(\mathfrak U;L^2_{\Div}(\mt,\R^2))$ for $\bfu\in L^2_{\Div}(\mt,\R^2)$, and
$\Phi(\bfu)\in L_2(\mathfrak U;W^{1,2}_{0,\Div}(\mt,\R^2))$ for $\bfu\in W^{1,2}_{0,\Div}(\mt,\R^2)$, together with
\begin{align}\label{eq:phi0}
\|\Phi(\bfu)-\Phi(\bfv)\|_{L_2(\mathfrak U;L^2_x)}&\leq\,c\|\bfu-\bfv\|_{L^2_x}\qquad\forall \bfu,\bfv\in L^2_{\Div}(\mt,\R^2),\\
\label{eq:phi1a}
\|\Phi(\bfu)\|_{L_2(\mathfrak U;W^{1,2}_x)}&\leq\,c\big(1+\|\bfu\|_{W^{1,2}_x}\big)\qquad\forall \bfu\in W^{1,2}_{0,\Div}(\mt,\R^2),\\
\label{eq:phi1b}
\|D\Phi(\bfu)\|_{L_2(\mathfrak U;\mathcal L( L^{2}_x;L^2_x))}&\leq\,c\qquad\forall \bfu\in L^{2}_{\Div}(\mt,\R^2).
\end{align}
If we are interested in higher regularity, some further assumptions are in place and we require additionally 
that $\Phi(\bfu)\in L_2(\mathfrak U;W^{2,2}(\mt,\R^2))$ for $\bfu\in W^{2,2}\cap W^{1,2}_{0,\Div}(\mt,\R^2)$, together with
\begin{align}\label{eq:phi2a}
&\|\Phi(\bfu)\|_{L_2(\mathfrak U;W^{2,2}_x)}\leq\,c\big(1+\|\bfu\|_{W^{1,4}_x}^2+\|\bfu\|_{W^{2,2}_x}\big)\qquad\forall \bfu\in W^{2,2}\cap W^{1,2}_{0,\Div}(\mt,\R^2),\\\label{eq:phi2b}
&\|D^2\Phi(\bfu)\|_{L_2(\mathfrak U;\mathcal L( L^{2}_x\times L^2_x;L^2_x))}\leq\,c\qquad\forall \bfu\in L^{2}_{\Div}(\mt,\R^2).
\end{align}

Assumption \eqref{eq:phi0} allows us to define stochastic integrals.
Given an $(\mathfrak F_t)$-adapted process $\bfu\in L^2(\Omega;C([0,T];L^2_{\Div}(\mt)))$, the stochastic integral $$t\mapsto\int_0^t\varPhi(\bfu)\,\dif W$$
is a well-defined process taking values in $L^2_{\Div}(\mt,\R^2)$; see \cite{PrZa} for a detailed construction. Moreover, we can multiply by test functions to obtain
 \begin{align*}
\bigg(\int_0^t \varPhi(\bfu)\,\dd W,\bfphi\bigg)_{L^2_x}=\sum_{j\geq 1} \int_0^t( \varPhi(\bfu) e_j,\bfphi)_{L^2_x}\,\dd\beta_j \qquad \forall\, \bfphi\in L^2(\mt,\R^2).
\end{align*}
Similarly, we can define stochastic integrals with values in $W^{1,2}_{0,\Div}(\mt,\R^2)$ and $W^{2,2}(\mt,\R^2)$, respectively, if $\bfu$ belongs to the corresponding class.


\subsection{The concept of solutions}
\label{subsec:solution}
In dimension two, pathwise uniqueness for analytically weak solutions is known under the assumption \eqref{eq:phi0}; we refer the reader for instance to Capi\'nski--Cutland \cite{CC}, Capi\'nski \cite{Ca}. Consequently, we may work with the definition of a weak pathwise solution.

\begin{definition}\label{def:inc2d}
Let $(\Omega,\mf,(\mf_t)_{t\geq0},\prst)$ be a given stochastic basis with a complete right-con\-ti\-nuous filtration and an $(\mf_t)$-cylindrical Wiener process $W$. Let $\bfu_0$ be an $\mf_0$-measurable random variable with values in $L^2_{\Div}(\mt,\R^2)$. Then $\bfu$ is called a \emph{weak pathwise solution} \index{incompressible Navier--Stokes system!weak pathwise solution} to \eqref{eq:SNS} with the initial condition $\bfu_0$ provided
\begin{enumerate}
\item the velocity field $\bfu$ is $(\mf_t)$-adapted and
$$\bfu \in C_{\mathrm loc}([0,\infty);L^2_{\diver}(\tor,\R^2))\cap L^2_{\mathrm loc}(0,\infty; W^{1,2}_{0,\Div}(\tor,\R^2))\quad\text{$\p$-a.s.},$$
\item the momentum equation
\begin{align*}
\int_{\mt}\bfu(t)&\cdot\bfvarphi\dx-\int_{\mt}\bfu_0\cdot\bfvarphi\dx
\\&=\int_0^t\int_{\mt}\bfu\otimes\bfu:\nabla\bfvarphi\dx\,\dif s+\mu\int_0^t\int_{\mt}\nabla\bfu:\nabla\bfvarphi\dx\,\dif s+\int_0^t\int_{\mt}\Phi(\bfu)\cdot\bfvarphi\dx\,\dif W
\end{align*}
holds $\p$-a.s. for all $\bfvarphi\in W^{1,2}_{0,\diver}(\mt,\R^2)$ and all $t\geq0$.
\end{enumerate}
\end{definition}

\begin{theorem}\label{thm:inc2d}
Suppose that $\Phi$ satisfies \eqref{eq:phi1a} and \eqref{eq:phi1b}. Let $(\Omega,\mf,(\mf_t)_{t\geq0},\prst)$ be a stochastic basis with a complete right-continuous filtration and an $(\mf_t)$-cylindrical Wiener process $W$. Let $\bfu_0$ be an $\mf_0$-measurable random variable such that $\bfu_0\in L^r(\Omega;L^2_{\mathrm{div}}(\mt,\R^2))$ for some $r>2$. Then there exists a unique weak pathwise solution to \eqref{eq:SNS} in the sense of Definition \ref{def:inc2d} with the initial condition $\bfu_0$.
\end{theorem}

We give the definition of a strong pathwise solution to \eqref{eq:SNS} which exists up to a stopping time $\mathfrak t$. The velocity field here belongs $\p$-a.s. to $C([0,\mathfrak t];W^{1,2}_{0,\diver}(\mt,\R^2))$.

\begin{definition}\label{def:strsol}
Let $(\Omega,\mf,(\mf_t)_{t\geq0},\prst)$ be stochastic basis with a complete right-continuous filtration and an $(\mf_t)$-cylindrical Wiener process $W$. Let $\bfu_0$ be an $\mf_0$-measurable random variable with values in $W^{1,2}_{0,\diver}(\mt,\R^2)$. The tuple $(\bfu,\mathfrak t)$ is called a \emph{local strong pathwise solution} \index{incompressible Navier--Stokes system!weak pathwise solution} to \eqref{eq:SNS} with the initial condition $\bfu_0$ provided
\begin{enumerate}
\item $\mathfrak t$ is a $\p$-a.s. strictly positive $(\mathfrak F_t)$-stopping time;
\item the velocity field $\bfu$ is $(\mf_t)$-adapted and
$$\bfu(\cdot\wedge \mathfrak t) \in C_{\mathrm loc}([0,\infty);W^{1,2}_{0,\diver}(\mt,\R^2))\cap L^2_{\mathrm loc}(0,\infty;W^{2,2}(\mt,\R^2)) \quad\text{$\p$-a.s.},$$
\item the momentum equation
\begin{align}\label{eq:mom}
\begin{aligned}
&\int_{\mt}\bfu(t\wedge \mathfrak t)\cdot\bfvarphi\dx-\int_{\mt}\bfu_0\cdot\bfvarphi\dx
\\&=-\int_0^{t\wedge\mathfrak t}\int_{\mt}(\bfu\cdot\nabla)\bfu\cdot\bfvarphi\dx\,\dif s+\mu\int_0^{t\wedge\mathfrak t}\int_{\mt}\Delta\bfu\cdot\bfvarphi\dx\,\dif s+\int_{\mt}\int_0^{t\wedge\mathfrak t}\Phi(\bfu)\cdot\bfvarphi\,\dif W\dx
\end{aligned}
\end{align}
holds $\p$-a.s. for all $\bfvarphi\in C^{\infty}_{c,\diver}(\mt,\R^2)$
and all $t\geq0$.
\end{enumerate}
\end{definition}
Note that \eqref{eq:mom} certainly implies the corresponding formulation in Definition \ref{def:inc2d}. The reverse implication is only true for analytically strong solutions.\\
We finally define what a maximal strong pathwise solution is.
\begin{definition}[Maximal strong pathwise solution]\label{def:maxsol}
Fix a stochastic basis with a cylindrical Wiener process and an initial condition as in Definition \ref{def:strsol}. A triplet $$(\bfu,(\mathfrak{t}_R)_{R\in\N},\mathfrak{t})$$ is a maximal strong pathwise solution to system \eqref{eq:SNS} provided

\begin{enumerate}
\item $\mathfrak{t}$ is a $\p$-a.s. strictly positive $(\mathfrak{F}_t)$-stopping time;
\item $(\mathfrak{t}_R)_{R\in\mn}$ is an increasing sequence of $(\mathfrak{F}_t)$-stopping times such that
$\mathfrak{t}_R<\mathfrak{t}$ on the set $[\mathfrak{t}<\infty]$, as well as
$\lim_{R\to\infty}\mathfrak{t}_R=\mathfrak t$ $\p$-a.s., and
\begin{equation}\label{eq:blowup}
\mathfrak t_R:=\inf \big\{t\in[0,\infty):\,\,\|\bfu(t)\|_{W^{1,2}_x}\geq R\big\}\quad \text{on}\quad [\mathfrak{t}<\infty] ,
\end{equation}
with the convention that $\mathfrak{t}_R=\infty$ if the set above is empty;
\item each tuple $(\bfu,\mathfrak{t}_R)$, for $R\in\mn$,  is a local strong pathwise solution in the sense  of Definition \ref{def:strsol}.
\end{enumerate}
\end{definition}
We talk about a global solution
if we have (in the framework of Definition \ref{def:maxsol}) $\mathfrak t=\infty$ $\p$-a.s.
The following result is shown in \cite{Mi}; see also \cite{GlZi} for a similar statement.
\begin{theorem}\label{thm:inc2d}
Suppose that \eqref{eq:phi0}--\eqref{eq:phi1b} hold and
 that $\bfu_0\in L^2(\Omega,W^{1,2}_{0,\diver}(\mt,\R^2))$. Then there is a
unique global maximal strong pathwise solution to \eqref{eq:SNS} in the sense  of Definition \ref{def:maxsol}.
\end{theorem}


\subsection{Finite elements}
We work with a standard finite element set-up for incompressible fluid mechanics; see e.g. \cite{GR}.
We denote by $\mathscr{T}_h$ a quasi-uniform subdivision  \cite{brenner1} of $\mt$ into triangles of maximal diameter $h>0$.   
For $K\subset \mt$ and $\ell\in \setN _0$ we denote by
$\mathscr{P}_\ell(K)$ the polynomials on $K$ of degree less than or equal
to $\ell$. 
Let us
characterize the finite element spaces $V^h(\mt,\R^2)$ and $P^h(\mt)$ as
\begin{align}\label{def:Vh}
  V^h(\mt,\R^2)&:= \set{\bfv_h \in W^{1,2}_0(\mt,\R^2)\,:\, \bfv_h|_{K}
    \in (\mathscr{P}_i(K))^2\quad\forall K\in \mathscr T_h} \qquad  (i\geq 1),\\ \label{def:Ph}
P^h(\mt)&:=\set{\pi_h \in L^{2}(\mt)/\R\,:\, \pi_h|_{K}
    \in \mathscr{P}_j(K)\quad\forall K\in \mathscr T_h} \qquad  (j \geq 0).
\end{align}
In order to guarantee stability of our approximation we relate $V^h(\mt,\R^2)$ and $P^h(\mt)$ by the discrete inf-sup condition, {\em i.e.}, there exists a positive constant
$C$ not depending on $h$ such that 
\begin{align*}
\sup_{\bfv_h\in V^h(\mt,\R^2)} \frac{\int_{\mt}\Div\bfv_h\,\pi_h\dx}{\|\nabla\bfv_h\|_{L^2_x}}\geq\,C\,\|\pi_h\|_{L^2_x}\quad\,\forall\pi_h\in P^h(\mt) \, .
\end{align*}
A well-known class of inf-sup stable pairings are the `conforming Stokes elements', with the simplest choice  $i=2$ in (\ref{def:Vh}) and $j=0$ in (\ref{def:Ph}); see {\em e.g.} \cite[Ch. 6]{brezzi1} or \cite[Rem. 3.4]{HR} for further admissible examples of pairings.

We define the space of discretely solenoidal finite element functions by 
\begin{align*}
  V^h_{\Div}(\mt,\R^2)&:= \bigg\{\bfv_h\in V^h(\mt,\R^2):\,\,\int_{\mt}\Div\bfv_h\,\,\pi_h\dx=0\quad\forall\pi_h\in P^h(\mt)\bigg\}.
\end{align*}
Let $\Pi_h:L^2(\mt,\R^2)\rightarrow V_{\Div}^h(\mt,\R^2)$ be the $L^2(\mt,\R^2)$-orthogonal projection onto $V_{\Div}^h(\mt,\R^2)$. The following results concerning the approximability of $\Pi_h$ are well-known  (see, for instance \cite[Lemma 4.3]{HR}): there is $c>0$ independent of $h$ such that we have
  \begin{align}
    \label{eq:stab}
 \|\bfv-\Pi_h \bfv\|_{L^2_x}+ h\|\nabla\bfv-\nabla\Pi_h \bfv\|_{L^2_x} &\leq
    \,c\,h \|\nabla
      \bfv\|_{L^2_x}
  \end{align}
for all $\bfv\in W^{1,2}_{0,\Div}(\mt,\R^2)$; moreover, the arguments in \cite[Section 4]{HR} together with standard interpolations arguments (see \emph{e.g.} \cite[Lemma A.2]{GR}) also imply for $\beta\in(0,1]$ that
  \begin{align}
    \label{eq:stab'}
\|\bfv-\Pi_h \bfv\|_{L^2_x}+  h\|\nabla\bfv-\nabla\Pi_h \bfv\|_{L^2_x} &\leq
    \,c\,h^{1+\beta} \|
      \bfv\|_{W^{1+\beta,2}_x}
  \end{align}
for all $\bfv\in W^{1+\beta,2}\cap W^{1,2}_{0,\Div}(\mt,\R^2)$. Similarly, if $\Pi_h^\pi:L^2(\mt)/\R\rightarrow P^h(\mt)$ denotes the $L^2(\mt)$-orthogonal projection onto $P^h(\mt)$, we have
\begin{align}
\label{eq:stabpi}
 \|p-\Pi_h^\pi p\|^2_{L^2_x} &\leq
    \,c h\, \|\nabla
      p\|_{L^2_x}
\end{align}
for all $p\in W^{1,2}(\mt)/\R$.

\section{Regularity of solutions}
In this section we analyse the regularity of the continuous solution as well as the associated pressure function. For various purposes we need the Helmholtz-projection $\mathcal P:L^p(\mt,\R^2)\rightarrow L^2_{\Div}(\mt,\R^2)$, for $1<p<\infty$, given by
\begin{align}\label{helmholz}
\mathcal P\bfphi:=\bfphi-\nabla\Delta_{\mathcal O}^{-1}\Div\bfphi.
\end{align}
Here $\Delta^{-1}_{\mathcal O}\Div$ is the solution operator to the equation
\begin{align*}
\Delta \mathfrak {h}=\Div\bfg\quad\text{in}\quad\mathcal O,\quad \nu_{\mathcal O}\cdot(\nabla \mathfrak h-\bfg)=0\quad\text{on}\quad\partial\mathcal O,
\end{align*}
where $\nu_{\mathcal O}$ denotes the unit normal of $\partial\mathcal O$.
Note that $\nabla\Delta^{-1}_{\mt}\Div$ satisfies (since $\partial\mt$ was assumed to be sufficiently smooth)
\begin{align}
\label{eq:Delta3}
\nabla\Delta^{-1}_{\mt}\Div&:W^{r,p}(\mt,\R^2)\rightarrow W^{r,p}(\mt,\R^2),
\end{align}
for all $p\in(1,\infty)$ and all $r\in\N$, where $W^{0,p}(\mt,\R^2)=L^p(\mt,\R^2)$; see \cite{ADN} for the case $r\in\N$ and \cite[Chapter IV]{Ga} for the case $r=0$. Clearly, \eqref{eq:Delta3} transfers to $\mathcal P$.

With the help of the Helmholtz projection we can define the Stokes operator as
\begin{align}\label{eq:A}
 \mathcal A:=\mathcal P\Delta:W^{2,p}\cap W^{1,p}_{0,\Div}(\mt,\R^2)\rightarrow L^p_{\Div}(\mt,\R^2).
 \end{align}
 Due to well-known estimates for the Stokes system there is $c>0$ such that
 \begin{align}\label{eq:Akp}
 \|\bfu\|_{W^{r+2,p}_x}\leq\,c\,\|\A \bfu\|_{W^{r,p}_x},\quad \bfu\in W^{r+2,p}\cap W^{1,p}_{0,\Div}(\mt,\R^2),
 \end{align}
for all $p\in(1,\infty)$ and all $r\in\N_0$, see, \emph{e.g.}, \cite[Thm. IV. 6.1.]{Ga}, which uses sufficient smoothness of $\partial\mt$.
 Moreover, there is a system of eigenfunctions to the Stokes operator
$(\bfu_k)\subset W^{1,2}_{0,\Div}(\mt,\R^2)$ with strictly positive
eigenvalues $(\lambda_k)$ such that $\lambda_k\rightarrow\infty$ as $k\rightarrow\infty$. It is possible to choose the $\bfu_k$'s such that the system
 $(\bfu_k)$ is orthonormal in $L^2(\mt,\R^2)$ and orthogonal in $W^{1,2}_0(\mt,\R^2)$. Finally, we can assume that the $\bfu_k$'s are sufficiently smooth due to the assumed smoothness of $\partial\mt$. Since $\A$ is positive, its root $\A^{1/2}$ is well-defined with domain $W^{1,p}_{0,\Div}(\Omega,\R^2)$, and we have
\begin{align}\label{eq:sorootcont1}
&\|\nabla\bfu\|_{L^p_x}\leq c\big\|\A^{1/2}\bfu\big\|_{L^p_x}\leq C\|\nabla\bfu\|_{L^p_x},\quad \bfu\in W^{1,p}_{0,\Div}(\mt,\R^2),\\
&\int_\mt \A^{1/2}\bfu\cdot\bfw\dx=\int_{\mt}\bfu\cdot \A^{1/2} \bfw\dx,\quad\bfu\in W^{1,p}_{0,\Div}(\mt,\R^2),\,\,\bfw\in W^{1,p'}_{0,\Div}(\mt,\R^2),\label{eq:sorootcont2}
\end{align}
where $c,C>0$; cf.~\cite{GaSiSo}.

\subsection{Estimates for the continuous solution}
In this section we derive crucial estimates for the
maximal strong pathwise solution from Definition \ref{def:maxsol}, which hold up to the stopping time $\mathfrak t_R$. Here $R>0$ is a fixed truncation parameter and $T>0$ an arbitrary but fixed time.
\begin{lemma}\label{lem:reg}
Let $(\Omega,\mf,(\mf_t)_{t\geq0},\prst)$ be a given stochastic basis with a complete right-con\-ti\-nuous filtration and an $(\mf_t)$-cylindrical Wiener process $W$.
 \begin{enumerate}
\item[(a)] Assume that $\bfu_0\in L^r(\Omega,L^{2}_{\Div}(\mt,\R^2))$ for some $r\geq2$ and that $\Phi$ satisfies \eqref{eq:phi0}. Then we have
\begin{align}\label{eq:W12}
\E\bigg[\bigg(\sup_{0\leq t\leq T}\|\bfu(t)\|_{L^2_x}^2+\int_0^{T}\|\nabla\bfu\|_{L^2_x}^2\dt\bigg)^{\frac{r}{2}}\bigg]\leq\,c\,\E\Big[1+\|\bfu_0\|_{L^2_x}^r\Big],
\end{align}
where $\bfu$ is the weak pathwise solution to \eqref{eq:SNS}; cf. Definition \ref{def:inc2d}.
\item[(b)] Assume that $\bfu_0\in L^r(\Omega,W^{1,2}_{0,\Div}(\mt,\R^2))$ for some $r\geq2$ and that $\Phi$ satisfies \eqref{eq:phi0}--\eqref{eq:phi1b}. Then we have
\begin{align}\label{eq:W22}\begin{aligned}
\E\bigg[\bigg(\sup_{0\leq t\leq T}\|\bfu(t\wedge \mathfrak t_R)\|^2_{W^{1,2}_x}&+\int_0^{T\wedge \mathfrak t_R}\|\bfu\|^2_{W^{2,2}_x}\dt\bigg)^{\frac{r}{2}}\bigg]\leq\,cR^{3r}\,\E\Big[1+\|\bfu_0\|_{W^{1,2}_x}^r\Big],
\end{aligned}
\end{align}
where $(\bfu,(\mathfrak t_R)_{R\in\N},\mathfrak t)$ is the maximal strong pathwise solution to \eqref{eq:SNS}; cf. Definition \ref{def:maxsol}.
\item[(c)] Assume that $\bfu_0\in L^r(\Omega,W^{2,2}(\mt,\R^2))\cap L^{5r}(\Omega,W^{1,2}_{0,\Div}(\mt,\R^2))$ for some $r\geq2$ and that \eqref{eq:phi0}--\eqref{eq:phi2b} holds. Then we have for all $\beta<1$
\begin{align}\label{eq:W32}
\begin{aligned}
\E\bigg[\bigg(\sup_{0\leq t\leq T}\|\bfu(t\wedge \mathfrak t_R)\|_{W^{1+\beta}_x}^2&+\int_0^{T\wedge \mathfrak t_R}\|\bfu\|^2_{W^{2+\beta,2}_x}\dt\bigg)^{\frac{r}{2}}\bigg]\\&\leq\,cR^{5r}\,\E\Big[1+\|\bfu_0\|_{W^{2,2}_x}^r+\|\bfu_0\|_{W^{1,2}_x}^{2r}\Big],
\end{aligned}
\end{align}
where $(\bfu,(\mathfrak t_R)_{R\in\N},\mathfrak t)$ is the maximal strong pathwise solution to \eqref{eq:SNS}; cf. Definition \ref{def:maxsol}.
\end{enumerate}
Here $c=c(r,T,\beta)>0$ is independent of $R$.
\end{lemma}
\begin{proof}
{\bf Part (a)} is the standard a priori estimate, which is a consequence of applying It\^{o}'s formula to $t\mapsto \|\bfu\|_{L^2_x}^2$.

For {\bf part (b)} we follow \cite{Mi}, where the solution to a truncated problem is considered. For $R>1$ and $\zeta\in C_c^\infty([0,1))$ with $0\leq \zeta\leq 1$ and $\zeta=1$ in $[0,1]$ we set $\zeta_R:=\zeta(R^{-1}\cdot)$ . 
Similar to Definition \ref{def:inc2d} we seek an
$(\mf_t)$-adapted stochastic process $\bfu^R$ with
$$\bfu^R \in C([0,T];L^2_{\diver}(\tor,\R^2))\cap L^2(0,T; W^{1,2}_{0,\diver}(\tor,\R^2))\quad\text{$\p$-a.s.}$$
such that
\begin{align}\label{eq:trun}
\begin{aligned}
\int_{\mt}\bfu^R(t)\cdot\bfvarphi\dx&=\int_{\mt}\bfu_0\cdot\bfvarphi\dx
+\int_0^t  \zeta_R(\|\nabla\bfu^{R}\|_{L^2_x}) \int_{\mt}\bfu^{R}\otimes\bfu^{R} :\nabla\bfphi\dx\,\dif s\\&-\mu\int_0^t\int_{\mt}\nabla\bfu^R:\nabla\bfvarphi\dx\,\dif s+\int_0^t\int_{\mt}\Phi(\bfu^R)\cdot\bfvarphi\dx\,\dif W
\end{aligned}
\end{align}
holds $\p$-a.s. for all $\bfvarphi\in W^{1,2}_{0,\diver}(\mt,\R^2)$ and all $t\in[0,T]$. Arguing as in
\cite[Lemma 3.7]{Mi} one can show that a unique global strong pathwise solution to \eqref{eq:trun} exists in the class $C([0,T];W^{1,2}_{0,\Div}(\mt,\R^2))$,\footnote{Different from the solution obtained in Theorem \ref{thm:inc2d} it can be constructed for any given deterministic $T>0$.}
 and that it satisfies
\begin{align}\label{eq:W22R}\begin{aligned}
\E\bigg[\sup_{0\leq t\leq T}\|\nabla\bfu^R(t)\|_{L^2_x}^2\dx&+\int_{0}^T\|\nabla^2\bfu^R\|_{L^2_x}^2\dt\bigg]\leq\,c(r,R,T).
\end{aligned}
\end{align}
 The proof of \eqref{eq:W22R} in \cite{Mi} is based on a Galerkin approximation which we mimick now in order to prove \eqref{eq:W22} and \eqref{eq:W32}. 
 
\textbf{1) Galerkin approximation.} Let $(\bfu_k)\subset W^{1,2}_{0,\Div}(\mt,\R^2)$ be
a system of eigenfunctions to the Stokes operator, cf.~\eqref{eq:A}.
  For $N\in\N$ let $\mathbb H^N:=\mathrm{span}\{\bfu_1,\dots,\bfu_N\}$, and consider the unique
 solution $\bfu^{R,N}$ to
 \begin{align}\label{eq:trunN}
\begin{aligned}
\int_{\mt}\bfu^{R,N}(t)\cdot\bfvarphi\dx&=\int_{\mt}\bfu_0\cdot\bfvarphi\dx
+\int_0^t\zeta_R(\|\nabla\bfu^{R,N}\|_{L^2_x})\int_{\mt}\bfu^{R,N}\otimes\bfu^{R,N} :\nabla\bfphi\dx\,\dif s\\&-\mu\int_0^t\int_{\mt}\nabla\bfu^{R,N}:\nabla\bfvarphi\dx\,\dif s+\int_0^t\int_{\mt}\Phi(\bfu^{R,N})\cdot\bfvarphi\dx\,\dif W
\end{aligned}
\end{align}
for all $\bfphi\in \mathbb H^N$.
By $\mathcal P_N$ we denote the $L^2(\mt,\R^2)$-projection onto $\mathbb H^N$. Problem \eqref{eq:trunN}
can be written as a system of SDEs with Lipschitz-continuous coefficients. Hence it is clear that there is a unique strong solution, \emph{i.e.}, an $(\mathfrak F_t)$-adapted process defined on $(\Omega,\mathfrak F,\mathbb P)$ with values in $C([0,T];\mathbb H^N)$ and moments of order $r$. Arguing as in \cite[Prop. 3.2]{Mi} one can prove that
as $N\rightarrow\infty$
\begin{align}\label{lim:N}
\sup_{0\leq t\leq T}\|\bfu^R(t)-\bfu^{R,N}(t)\|_{L^2_x}^2&+\int_{0}^T\|\nabla(\bfu^R-\bfu^{R,N})\|_{L^2_x}^2\dxt\rightarrow 0
\end{align}
in probability. Applying It\^{o}'s formula to
$t\mapsto \|\bfu^{R,N}\|_{L^2_x}^2$ and using the cancellation of the convective term
 one can prove for $r\geq2$
\begin{align}\label{lim:N2}
\E\bigg[\bigg(\sup_{0\leq t\leq T}\|\bfu^{R,N}(t)\|_{L^2_x}^2+\int_{0}^T\|\nabla\bfu^{R,N}\|_{L^2_x}^2\dt\bigg)^{\frac{r}{2}}\bigg]\leq\,c\,\E\Big[1+\|\bfu_0\|_{L^2_x}^r\Big],
\end{align}
where $c=c(r,T)$ is independent of $N$ and $R$.

\textbf{2) Proof of \eqref{eq:W22}.} By construction we have $\mathcal A\bfu^{R,N}\in C([0,T];\mathbb H^N)$ $\p$-a.s. such that we can apply
It\^{o}'s formula to $t\mapsto (\bfu^{R,N}(t),\A \bfu^{R,N}(t))_{L^2_x}$ and use \eqref{eq:trunN}. This yields
using $\bfu^{R,N}|_{\partial\mt}=0$
\begin{align}
\nonumber
\|\nabla\bfu^{R,N}(t)\|_{L^2_x}^2&=- \big(\bfu^{R,N}(t),\Delta \bfu^{R,N}(t)\big)_{L^2_x}=- \big(\bfu^{R,N}(t),\A \bfu^{R,N}(t)\big)_{L^2_x}\\\nonumber&=\|\mathcal P_N\nabla\bfu_0\|^2_{^2_x}
+2\int_0^t   \zeta_R(\|\nabla\bfu^{R,N}\|_{L^2_x})\big((\bfu^{R,N}\cdot\nabla)\bfu^{R,N},\A\bfu^{R,N}\big)_{L^2_x}\,\dif s\\\label{eq:3005}
&\quad-2\mu\int_0^t\|\A\bfu^{R,N}\|_{L^2_x}^2\,\dif s+2\sum_{k=1}^N\int_0^t\big(\Phi(\bfu^{R,N})e_k,\A\bfu^{R,N}\big)_{L^2_x}\,\dif \beta_k\\
&\quad+\sum_{k=1}^N\lambda_k\int_0^t\big(\Phi(\bfu^{R,N})e_k,\bfu_k\big)_{L^2_x}^2\,\dif s\nonumber\\
&=:\mathrm{I}^N(t)+\dots+\mathrm{V}^N(t)\nonumber
\end{align}
$\mathbb P$-a.s. for all $t\in[0,T].$
We estimate now the terms
$\mathrm{II}^N$, $\mathrm{IV}^N$ and $\mathrm{V}^N$. First of all, we have by definition of $\zeta_R$
\begin{align*}
\mathrm{II}^N(t)
&\leq 2\int_0^t\zeta_R(\|\nabla\bfu^{R,N}\|_{L^2_x})\|\bfu^{R,N}\|_{L^4_x}\|\nabla\bfu^{R,N}\|_{L^4_x}\|\A\bfu^{R,N}\|_{L^2_x}\dif s\\
&\leq 2\int_0^t\zeta_R(\|\nabla\bfu^{R,N}\|_{L^2_x})\|\bfu^{R,N}\|_{L^2_x}^{\frac{1}{2}}\|\nabla\bfu^{R,N}\|_{L^2_x}\|\A\bfu^{R,N}\|^{\frac{3}{2}}_{L^2_x}\dif s\\
&\leq cR^{3/2}\int_0^t\|\A\bfu^{R,N}\|^{\frac{3}{2}}_{L^2_x}\dif s\leq \delta\int_0^t\|\A\bfu^{R,N}\|^{2}_{L^2_x}\dif s+c(\delta)R^6,
\end{align*}
where $\delta>0$ is arbitrary.
Moreover, 
we obtain by definition of $\bfu_k$ and
using \eqref{eq:sorootcont2} (and recalling that $\Phi(\bfu^{R,N})e_k\in W^{1,2}_{0,\Div}(\mt,\R^2)$ for all $k\in\N$ by assumption)
\begin{align*}
\mathrm{V}^N(t)&=\sum_{k=1}^N\int_0^t\big(\Phi(\bfu^{R,N})e_k,\sqrt{\lambda_k}\bfu_k\big)_{L^2_x}^2\,\dif s=\sum_{k=1}^N\int_0^t\big(\Phi(\bfu^{R,N})e_k, \A^{1/2}\bfu_k\big)^2\,\dif s
\\&=\sum_{k=1}^N\int_0^t\big(\mathcal A^{1/2}\Phi(\bfu^{R,N})e_k,\bfu_k\big)_{L^2_x}^2\,\dif s.
\end{align*}
Furthermore, since $\|\bfu_k\|_{L^2_x}=1$,
\begin{align*}
\mathrm{V}^N(t)&\leq\sum_{k\geq1}\int_0^t\|\mathcal A^{1/2}\Phi(\bfu^{R,N})e_k\|_{L^2_x}^2\|\bfu_k\|^2_{L^2_x}\,\dif s
\leq\,c\sum_{k\geq1}\int_0^t\|\nabla\Phi(\bfu^{R,N})e_k\|_{L^2_x}^2\,\dif s\\
&=c\int_0^t\|\Phi(\bfu^{R,N})\|^2_{L_2(\mathfrak U;W^{1,2}_x)}\,\dif s\leq\,c\int_0^t\big(1+\|\bfu^{R,N}\|^2_{W^{1,2}_x}\big)\,\dif s,
\end{align*}
using \eqref{eq:phi1a} in the last step.
The expectation of the right-hand side is bounded by \eqref{lim:N2}.
Finally, by Burkholder-Davis-Gundy inequality and  \eqref{eq:phi0},
\begin{align*}
\E\bigg[\bigg(\sup_{0\leq t\leq T}|\mathrm{IV}(t)|\bigg)^{\frac{r}{2}}\bigg]&\leq \E\bigg[\bigg(\sup_{0\leq t \leq T}\Big|\int_0^t\sum_{k=1}^N\big(\Phi(\cdot,\bfu^{R,N})e_k
,\mathcal A\bfu^{R,N}\big)_{L^2_x}\,\dd\beta_k\Big|\bigg)^{\frac{r}{2}}\bigg]\\
&\leq c\,\E\bigg[\bigg(\sum_{k\geq1}\int_0^T\big( \Phi(\cdot,\bfu^{R,N})e_k
\cdot\mathcal A\bfu^{R,N}\big)_{L^2_x}^2\dt\bigg)^{\frac r4}\bigg]\\
&\leq c\,\E\bigg[\bigg(\sum_{k\geq1}\int_0^T
\|\Phi_k(\bfu^{R,N})e_k\|_{L^2_x}^2\|\mathcal A\bfu^{R,N}\|^2_{L^2_x}\dt\bigg)^{\frac r4}\bigg]\\
&\leq c\,\E\bigg[\bigg(\int_0^T
\big(1+\|\bfu^{R,N}\|_{L^2_x}^2\big)\|\mathcal A\bfu^{R,N}\|^2_{L^2_x}\dt\bigg)^{\frac r4}\bigg]\\
&\leq c(\delta)\,\E\bigg[\bigg(1+\sup_{0\leq t\leq T}\|\bfu^{R,N}\|_{L^2_x}^2\bigg)^{\frac{r}{2}}\bigg]+ \delta\,\E\bigg[\bigg(\int_0^T\|\mathcal A\bfu^{R,N}\|_{L^2_x}^2\dt\bigg)^{\frac{r}{2}}\bigg]\\
&\leq c(\delta)+\delta\,\E\bigg[\bigg(\int_0^T\|\mathcal A\bfu^{R,N}\|_{L^2_x}^2\dt\bigg)^{\frac{r}{2}}\bigg]
\end{align*}
using \eqref{lim:N2}, where again $\delta>0$ is arbitrary.
Choosing $\delta$ small enough and using \eqref{eq:Akp}
we conclude that
\begin{align}\label{uniform:NR}
\begin{aligned}
\E\bigg[\bigg(\sup_{0\leq t\leq T}\int_{\mt}\|\nabla\bfu^{R,N}(t)\|_{L^2_x}^2&+\int_{0}^T\|\nabla^2\bfu^{R,N}\|_{L^2_x}^2\dt\bigg)^{\frac{r}{2}}\bigg]\\&\leq\,cR^{3r}\,\E\Big[1+\|\bfu_0\|_{W^{1,2}_x}^r\Big]\,,
\end{aligned}
\end{align}
uniformly in $N$. This implies that $(\bfu^{R,N})_{N\in\N}$ is a bounded sequence in the function space generated by the left-hand side of \eqref{uniform:NR}. After taking a subsequence we obtain a limit object $\bfu^R$ which is the unique global strong solution to \eqref{eq:trun} recalling
 \eqref{lim:N}. Furthermore, we can pass to the limit $N\rightarrow\infty$ and obtain a corresponding estimate for $\bfu^R$ due to lower semi-continuity of the involved functionals.
Since $\bfu^R(\cdot\wedge\mathfrak t_R)=\bfu(\cdot\wedge \mathfrak t_R)$ we obtain \eqref{eq:W22}.

\textbf{3) Proof of \eqref{eq:W32}.} The verification of {\bf part (c)} proceeds in two steps. In the first step we show an improved version of \eqref{uniform:NR}.
Applying It\^{o}'s formula to the mapping $t\mapsto \|\nabla\bfu^{R,N}(t)\|_{L^2_x}^2\big(\bfu^{R,N}(t),\A \bfu^{R,N}(t)\big)_{L^2_x}$,
equation \eqref{eq:3005} yields
\begin{align*}
\|\nabla\bfu^{R,N}(t)\|_{L^2_x}^4&=\|\mathcal P_N\nabla\bfu_0\|_{L^2_x}^4-4\mu\int_0^t\|\nabla\bfu^{R,N}\|_{L^2_x}^2\|\A\bfu^{R,N}\|_{L^2_x}^2\,\dif s\\&
\quad+4\int_0^t   \zeta_R(\|\nabla\bfu^{R,N}\|_{L^2_x})\|\nabla\bfu^{R,N}\|_{L^2_x}^2\big((\bfu^{R,N}\cdot\nabla)\bfu^{R,N},\A\bfu^{R,N}\big)_{L^2_x}\,\dif s\\&\quad+4\sum_{k=1}^N\int_0^t\|\nabla\bfu^{R,N}\|_{L^2_x}^2\big(\Phi(\bfu^{R,N})e_k,\A\bfu^{R,N}\big)_{L^2_x}\,\dif \beta_k\\
&\quad+2\sum_{k=1}^N\lambda_k\int_0^t\|\nabla\bfu^{R,N}\|_{L^2_x}^2\big(\Phi(\bfu^{R,N})e_k,\bfu_k\big)_{L^2_x}^2\,\dif s\\
&\quad+2\sum_{k=1}^N\int_0^t\big(\Phi(\cdot,\bfu^{R,N})e_k
,\mathcal A\bfu^{R,N}\big)_{L^2_x}^2\dt.
\end{align*}
Following now step by step the arguments from the proof of \eqref{uniform:NR} above we arrive at
\begin{align*}
\E\bigg[\bigg(\sup_{0\leq t\leq T}\|\nabla\bfu^{R,N}(t)\|_{L^2_x}^4&+\int_0^T\|\nabla\bfu^{R,N}\|_{L^2_x}^2\|\nabla^2\bfu^{R,N}\|_{L^2_x}^2\dt\bigg)^{\frac{r}{2}}\bigg]\\&\leq\,cR^{3r}\,\E\Big[1+\|\bfu_0\|_{W^{1,2}_x}^{2r}\Big].
\end{align*}
Again we can pass to the limit in $N$ obtaining
\begin{align}\label{uniform:NR'}
\begin{aligned}
\E\bigg[\bigg(\sup_{0\leq t\leq T}\|\nabla\bfu^{R}(t)\|_{L^2_x}^4&+\int_0^T\|\nabla\bfu^{R}\|_{L^2_x}^2\|\nabla^2\bfu^{R}\|_{L^2_x}^2\dt\bigg)^{\frac{r}{2}}\bigg]\\&\leq\,cR^{3r}\,\E\Big[1+\|\bfu_0\|_{W^{1,2}_x}^{2r}\Big].
\end{aligned}
\end{align}
Now we turn to the proof of  \eqref{eq:W32} stated in {\bf part (c)} for which we use
the mild formulation of \eqref{eq:trun}.

($\mathbf{c}_1$)
Due to the regularity proved in \eqref{uniform:NR} and \eqref{uniform:NR'}, \cite[Proposition
F.0.5, (i)]{PrRo} applies and we can write
\begin{align*}
\bfu^{R}(t)&=e^{-t\A}\bfu_0+\int_0^te^{-(t-s)\A}\bfg_R\,\dif s+\int_0^t e^{-(t-s)\A}\Phi(\bfu^{R})\,\dd W,\\
\text{where}\quad\bfg_R:&= \zeta_R(\|\nabla\bfu^{R}\|_{L^2_x})\mathcal P[(\bfu^{R}\cdot\nabla)\bfu^{R}].
\end{align*}
Here $(e^{-t\A})_{t\geq0}$ denotes the analytic semigroup on $L^2_{\Div}(\mt,\R^2)$ generated by the Stokes operator $\A$.
Setting
\begin{align*}
\bfY^{R}(t)&:=e^{-t\A}\bfu_0+\int_0^te^{-(t-s)\A} \bfg_R\,\dif s,\\
\bfZ^{R}(t)&:=\int_0^t e^{-(t-s)\A}\Phi(\bfu^{R})\,\dd W,
\end{align*}
we consider now the deterministic and stochastic contribution separately. We note that $\bfY^R$ is the unique solution to a deterministic Stokes problem with initial datum $\bfu_0$ and forcing $\bfg_R$, whereas $\bfZ^R$ solves a stochastic Stokes problem
with homogeneous initial datum and diffusion coefficient $\Phi(\bfu^R)$ -- both equipped with homogeneous Dirichlet boundary conditions.

By Ladyshenskaya's inequality we have
\begin{align*}
\E\bigg[\bigg(\int_0^T\|\bfg_R\|_{L^2_x}^2\dt\bigg)^{\frac{r}{2}}\bigg]&\leq\E\bigg[\bigg(\int_0^T\|\bfu^R\|_{L^4_x}^2\|\nabla\bfu^R\|_{L^4_x}^2\dt\bigg)^{\frac{r}{2}}\bigg]\\
&\leq\,c\,\E\bigg[\bigg(\int_0^T\|\bfu^R\|_{L^2_x}\|\nabla\bfu^R\|_{L^2_x}^2\|\nabla^2\bfu^R\|_{L^2_x}\dt\bigg)^{\frac{r}{2}}\bigg]\\
&\leq\,cR^{3r}\,\E\Big[1+\|\bfu_0\|_{W^{1,2}_x}^{2r}\Big],
\end{align*}
where we used \eqref{uniform:NR'} in the last step.

($\mathbf{c}_2$) Interpolating $W^{1/2,2}(0,T;W^{1,2}(\mt,\R^2))$ between
$W^{1,2}(0,T;L^2(\mt,\R^2))$ and $L^2(0,T;W^{2,2}(\mt,\R^2))$ and applying $\mathbb P$-a.s. classical estimates for the Stokes system yields
\begin{align}\label{eq:joerna}
\begin{aligned}
\E\bigg[\bigg(\|\bfY^R\|_{W^{1/2}(0,T;W^{1,2}_x)}^2\dt\bigg)^{\frac{r}{2}}\bigg]&\leq\E\bigg[\bigg(\|\bfY^R\|_{W^{1,2}(0,T;L^{2}_x)}^2+\|\bfY^R\|_{L^{2}(0,T;W^{2,2}_x)}^2\dt\bigg)^{\frac{r}{2}}\bigg]\\
&\leq\,c\,\E\bigg[\bigg(\int_0^T\|\bfg_R\|_{L^2_x}^2\dt\bigg)^{\frac{r}{2}}\bigg]\leq\,cR^{3r}\,\E\Big[1+\|\bfu_0\|_{W^{1,2}_x}^{2r}\Big].
\end{aligned}
\end{align}
($\mathbf{c}_3$) For $\bfZ^R$ we apply the recent results from \cite[Theorems 25 and 28]{Joern} proving for any $\sigma<1$
\begin{align}\label{eq:joernb}
\begin{aligned}
\E\bigg[\|\bfZ^R\|^2_{C^{\sigma/2}([0,T];L^2_x)}+\|\bfZ^R\|_{W^{\sigma/2,2}(0,T;W^{1,2}_x)}^2\bigg]&\leq\,c\,\E\bigg[1+\sup_{0\leq t\leq T}\|\bfu^R\|^{4}_{W^{1,2}_x}\bigg]\\
&\leq\,c\,\E\bigg[1+\|\bfu_0\|^{4}_{W^{1,2}_x}\bigg]
\end{aligned}
\end{align}
using also \eqref{eq:phi1a} and \eqref{uniform:NR'}. Combining \eqref{eq:joerna} and \eqref{eq:joernb} and recalling that $\bfu^R$ is the sum of $\bfY^R$ and $\bfZ^R$ gives
\begin{align}\label{eq:joern}
\begin{aligned}
\E\bigg[\|\bfu^R\|^2_{C^{\sigma/2}([0,T];L^2_x)}+\|\bfu^R\|_{W^{\sigma/2,2}(0,T;W^{1,2}_x)}^2\bigg]
&\leq\,c\,\E\bigg[1+\|\bfu_0\|^{4}_{W^{1,2}_x}\bigg].
\end{aligned}
\end{align}

($\mathbf{c}_4$) Due to our assumption on the noise from \eqref{eq:phi2a} we know that $\Phi(\bfu^R)e_k$, with $k\in\N$, belongs to the domain of the Stokes operator
such that we can write
\begin{align*}
\A\bfZ^{R}(t)=\int_0^t e^{-(t-s)\A}\A\Phi(\bfu^{R})\,\dd W.
\end{align*}
We conclude that $\A\bfZ^R$ is the unique weak pathwise solution to the stochastic Stokes problem with zero initial datum, homogeneous boundary conditions and diffusion coefficient $\A\Phi(\bfu^R)$. It is standard to derive for $r\geq2$ the estimate
\begin{align*}\E\bigg[\bigg(\sup_{0\leq t\leq T}\|\A\bfZ^{R}\|^2_{L^{2}_x}&+\int_0^T\|\nabla\A\bfZ^R\|^2_{L^2_x}\,\dif s\bigg)^{\frac{r}{2}}\bigg]\\
&\leq\,c\,\E\bigg[\bigg(\int_0^T\|\A\Phi(\bfu^{R})\|^2_{L_2(\mathfrak U;L^2_x)}\,\dd s\bigg)^{\frac{r}{2}}\bigg]\\
&\leq\,c\,\E\bigg[\bigg(\int_0^T\|\Phi(\bfu^{R})\|^2_{L_2(\mathfrak U;W^{2,2}_x)}\,\dd s\bigg)^{\frac{r}{2}}\bigg]\\
&\leq\,c\bigg[\bigg(\int_0^{t}\big(1+\|\bfu^{R}\|_{W^{1,2}_x}^2\|\bfu^{R}\|^2_{W^{2,2}_x}+\|\bfu^{R}\|^2_{W^{2,2}_x}\big)\,\dif s\bigg)^{\frac{r}{2}}\bigg],
\end{align*}
applying It\^{o}'s formula to $t\mapsto \|\A\bfu^R\|_{L^2_x}^2$ and using Burkholder-Davis-Gundy inequality (and \eqref{eq:phi2a} in the last step). The properties of the Stokes operator from \eqref{eq:Akp} yield
\begin{align*}\E\bigg[\bigg(\sup_{0\leq t\leq T}\|\bfZ^{R}\|^2_{W^{2,2}_x}&+\int_0^T\|\bfZ^R\|^2_{W^{3,2}_x}\,\dif s\bigg)^{\frac{r}{2}}\bigg]\\
&\leq\,c\bigg[\bigg(\int_0^{t}\big(1+\|\bfu^{R}\|_{W^{1,2}_x}^2\|\bfu^{R}\|^2_{W^{2,2}_x}+\|\bfu^{R}\|^2_{W^{2,2}_x}\big)\,\dif s\bigg)^{\frac{r}{2}}\bigg].
\end{align*}
($\mathbf{c}_5$) To sharpen the estimates for $\bfY^R$ is slightly more involved as the convective term $\bfg_R$ does not lie in the domain of the Stokes operator since it does not necessarily have a zero trace.
We can choose $p<2$ such that the embedding
$W^{1,p}(\mt)\hookrightarrow W^{\sigma,2}(\mt)$ holds.
We obtain by continuity of $\mathcal P$, cf.~\eqref{eq:Delta3},
\begin{align*}
\|\bfg_R\|_{W^{\sigma,2}_x}&\leq\,c\,\|\bfg_R\|_{W^{1,p}_x}\leq\,c
\|(\bfu^{R}\cdot\nabla)\bfu^{R}\|_{W^{1,p}_x}\\
&\leq\,c\|\nabla\bfu^R\|_{L^{2p}_x}^2+c\|\bfu^R\|_{L^q_x}\|\nabla^2\bfu_R\|_{L^2_x}\leq\,c\|\nabla\bfu^R\|_{L^2_x}\|\nabla^2\bfu_R\|_{L^2_x},
\end{align*}
where we used H\"older's inequality with exponents $2/p$ and $q:=2/(2-p)$
as well as Sobolev's embedding $W^{1,2}(\mt,\R^2)\hookrightarrow L^q(\mt,\R^2)$ and Ladyshenskaya's inequality. By \eqref{uniform:NR'} we conclude that
\begin{align}\label{eq:1409a}
\bfg_R\in L^2(0,T;W^{\sigma,2}(\mt,\R^2))\quad\mathbb P\text{-a.s.}
\end{align}
We argue now similarly for the temporal regularity of order $\sigma/2$ obtaining for any $\sigma'\in(\sigma,1)$
\begin{align*}
\|\bfg_R\|_{W^{\sigma/2,p}(0,T;L^p_x)}^p&\leq \,c\int_0^T\int_0^T\frac{\|\bfu^R(t)\nabla\bfu^R(t)-\bfu^R(s)\nabla\bfu^R(s)\|_{L^p_x}^p}{|t-s|^{1+p\sigma/2}}\dt\,\dif s\\
&\leq \,c\int_0^T\int_0^T\bigg(\frac{\|\bfu^R(t)-\bfu^R(s)\|_{L^2_x}}{|t-s|^{\sigma'/2}}\|\nabla\bfu^R(t)\|_{L^q_x}\bigg)^p\frac{\dt\,\dif s}{|t-s|^{1+\frac{(\sigma-\sigma')p}{2}}}\\
&\quad+ \,c\int_0^T\int_0^T\bigg(\frac{\|\bfu^R(s)\|_{L^q_x}\|\nabla\bfu^R(t)-\nabla\bfu^R(s)\|_{L^2_x}}{|t-s|^{\sigma/2}}\bigg)^p\frac{\dt\,\dif s}{|t-s|}\\
&\leq \,c\|\bfu^R\|^p_{C^{\sigma'/2}([0,T];L^2_x)}\int_0^T\|\nabla\bfu^R(t)\|_{L^q_x}^p\dt\\
&\quad+ \,c\sup_{0\leq s\leq t}\|\bfu^R(s)\|_{L^q_x}^p\int_0^T\int_0^T\frac{\|\nabla\bfu^R(t)-\nabla\bfu^R(s)\|_{L^2_x}^p}{|t-s|^{1+p\sigma/2}}\dt\,\dif s\\
&\leq \,c\|\bfu^R\|^p_{C^{\sigma'/2}([0,T];L^2_x)}\int_0^T\big(1+\|\bfu^R(t)\|_{W^{2,2}_x}^2\big)\dt\\
&\quad+ \,c\sup_{0\leq s\leq t}\|\bfu^R(s)\|_{W^{1,2}_x}^p\|\bfu^R\|_{W^{\sigma/2,p}(0,T;W^{1,2}_x)}^p\\
&\leq\,c\bigg(\|\bfu^R\|^2_{C^{\sigma'/2}([0,T];L^2_x)}+\|\bfu^R\|_{W^{\sigma'/2,2}(0,T;W^{1,2}_x)}^2+1\bigg)\\
&\quad+c\bigg(\,\sup_{0\leq s\leq t}\|\bfu^R(s)\|_{W^{1,2}_x}^2+\int_0^T\|\bfu^R\|_{W^{2,2}_x}^2\dt\bigg)^q.
\end{align*}
The expectation of the right-hand side is bounded using \eqref{eq:W32} and \eqref{eq:joern}; in particular, for any $\sigma<1$
\begin{align}\label{eq:1409b}
\bfg_R\in W^{\sigma/2,2}(0,T;L^{2}(\mt,\R^2))\quad\mathbb P\text{-a.s.}
\end{align}
using the embedding decreasing the value of $\sigma$ and using $W^{\sigma/2,p}(0,T)\hookrightarrow W^{\sigma'/2,2}(0,T)$ for an appropriate choice of $\sigma>\sigma'$ and $p<2$.
By \eqref{eq:1409a} and \eqref{eq:1409b} classical results on the Stokes system (see \cite[Thm. 15]{Soa}) and interpolation yield
\begin{align*}
\bfY^R\in W^{1+\sigma/2}(0,T;L^{2}(\mt,\R^2))\cap  L^2(0,T;W^{2+\sigma,2}(\mt,\R^2))\quad\mathbb P\text{-a.s.} 
\end{align*}
and thus, again by interpolation and appropriate choice of $\sigma\in(\beta,1)$ and the embedding $W^{\alpha,2}(0,T)\hookrightarrow L^\infty(0,T)$ for $\alpha>1/2$,
\begin{align}\label{eq:0809b}
\bfY^R\in L^\infty(0,T;W^{1+\beta,2}(\mt,\R^2))\cap  L^2(0,T;W^{2+\beta,2}(\mt,\R^2))\quad\mathbb P\text{-a.s.} 
\end{align}
together with
\begin{align*}
\sup_{0\leq t\leq T}&\|\bfY^{R}\|_{W^{1+\beta,2}_x}^2+\int_0^T\|\bfY^{R}\|_{W^{2+\beta,2}_x}^2\,\dd s\\&\leq \,c\bigg[\|\bfu_0\|_{W^{1+\sigma,2}_x}^2+\|\bfg_R\|^2_{W^{\sigma/2,2}_t(L^2_x)}+\int_0^T\| \bfg_R\|_{W^{\sigma,2}_x}^2\,\dd s\bigg]\quad\mathbb P\text{-a.s.} 
\end{align*}
Combining the estimates for $\bfY^{R}$ and $\bfZ^{R}$, choosing $\kappa$ sufficiently small and
using \eqref{uniform:NR} and \eqref{uniform:NR'} we arrive at
\begin{align*}
\E\bigg[\bigg(\sup_{0\leq t\leq T}\|\bfu^{R}(t)\|_{W^{1+\sigma,2}_x}^2\dx&+\int_0^{T}\|\bfu^{R}\|^2_{W^{2+\sigma,2}_x}\dt\bigg)^{\frac{r}{2}}\bigg]\\&\leq\,cR^{5r}\,\E\Big[1+\|\bfu_0\|_{W^{2,2}_x}^r+\|\bfu_0\|_{W^{1,2}_x}^{2r}\Big].
\end{align*}
uniformly in $R$.
\end{proof}

\begin{remark}\label{rem:reg}
{\bf 1.}
It seems not possible to prove Lemma \ref{lem:reg} (c) for $\beta\geq1$, see \eqref{eq:0809b}. In fact, even for the deterministic Stokes system high regularity is only possible if the forcing is regular in space \emph{and} time or belongs to the domain of the Stokes operator. Since neither is true for the convective term $\mathcal P(\bfu\cdot\nabla)\bfu$ (its temporal regularity is restricted by that of the driving Wiener process) we conjecture that the spatial regularity from \ref{lem:reg} (c) is optimal. Interestingly, this is just enough to prove an optimal convergence rate for the discretisation of \eqref{eq:SNS} in Theorem \ref{thm:main}.\\
{\bf 2.} Using a recent result from \cite{Joern} we can show
that the gradient of velocity field and hence the convective
has a fractional time derivative of order $\beta/2<1/2$. This is optimal in view of the limited regularity of the driving Wiener process in the momentum equation. It is classical for deterministic parabolic equations (see \cite{Soa} for the Stokes equations and \cite{Sob} for the heat equation) that the solution gains
two spatial and one temporal derivatives compared to the right-hand side. Hence the regularity of the latter has to be measured in space and time with respect to the parabolic scaling; pure space regularity does not transfer unless additional assumptions are in place such that we can only hope for $2+\beta$ spatial derivatives.
\end{remark}

\subsection{Regularity of the pressure}\label{sec:stochpress}
Since we will be working with discretely diver\-gence-free function spaces in the finite-element analysis for \eqref{txdiscrpi} in Section \ref{sec:error}, it is inevitable to introduce the pressure function. 
Note that the strong formulation of the momentum equation
in \eqref{eq:mom} even allows test functions from the class $L^2_{\Div}(\mt,\R^2)$ (using a standard smooth approximation argument), \emph{i.e.}, functions which do not have zero traces on $\partial\mt$. Hence
for $\bfphi\in C^\infty_c(\mt,\R^2)$ we can insert $$\mathcal P\bfphi=\bfphi-\nabla\Delta_{\mathcal O}^{-1}\Div\bfphi$$
with the Helmholz projection $\mathcal P$; cf. \eqref{helmholz}. 
We obtain
\begin{align}
\nonumber
\int_{\mt}\bfu(t\wedge\mathfrak t_R)\cdot\bfvarphi\dx &-\int_0^{t\wedge\mathfrak t_R}\int_{\mt}\mu\Delta \bfu\cdot\bfphi\dxs+\int_0^{t\wedge\mathfrak t_R}\int_{\mt}(\bfu\cdot\nabla)\bfu\cdot\bfphi\dxs\\
\label{eq:pressure}&=\int_{\mt}\bfu(0)\cdot\bfvarphi\dx
+\int_0^{t\wedge\mathfrak t_R}\int_{\mt}\pi\,\Div\bfphi\dxs
\\
\nonumber&+\int_{\mt}\int_0^{t\wedge\mathfrak t_R}\Phi(\bfu)\,\dd W\cdot \bfvarphi\dx,
\end{align}
where
\begin{align*}
\pi&=-\Delta^{-1}_{\mt}\Div\big((\bfu\cdot\nabla)\bfu\big).
\end{align*}

In the following we will analyse how the regularity of $\bfu$ transfers to $\pi$, where again
$R>0$ is a fixed truncation parameter and $T>0$ an arbitrary but fixed time.
\begin{lemma}\label{lem:pressure}
\begin{enumerate}
\item[(a)] Under the assumptions of Lemma \ref{lem:reg} (b) we have
\begin{align*}
\E\bigg[\bigg(\int_0^{T\wedge \mathfrak t_R}\|\pi\|_{W^{1,2}_x}^2\dt\bigg)^{\frac{r}{4}}\bigg]\leq\,cR^{3r}\,\E\Big[1+\|\bfu_0\|_{W^{1,2}_x}^r\Big].
\end{align*}
\item[(b)] Under the assumptions of Lemma \ref{lem:reg} (c) we have
\begin{align*}
\E\bigg[\bigg(\int_0^{T\wedge \mathfrak t_R}\|\pi\|_{W^{2,2}_x}^2\dt\bigg)^{\frac{r}{4}}\bigg]\leq\,cR^{5r}\,\E\Big[1+\|\bfu_0\|_{W^{2,2}_x}^r\Big].
\end{align*}
\end{enumerate}
Here $c=c(r,T)>0$ is independent of $R$.
\end{lemma}
\begin{proof}
Ad~{\bf (a)}. Arguing as in \cite[Corollary 2.5]{BrDo} and using \eqref{eq:Delta3} we obtain
\begin{align*}
\E\bigg[\bigg(\int_0^{T\wedge \mathfrak t_R}\|\pi\|_{W^{1,2}_x}^2\dt\bigg)^{\frac{r}{4}}\bigg]
&\leq\,c\,\E\bigg[\bigg(1+\sup_{t\in[0, T\wedge \mathfrak t_R]}\|\bfu\|_{W^{1,2}_x}^2+\int_0^{T\wedge \mathfrak t_R}\|\bfu\|_{W^{2,2}_x}^2\dt\bigg)^{\frac{r}{2}}\bigg]
\end{align*}
Consequently, Lemma \ref{lem:reg} (b) implies (a).

Ad~{\bf (b)}. Using \eqref{eq:Delta3} we have for $p>2$ close to 2 and $q:=\frac{2p}{p-2}$
\begin{align*}
\E\bigg[\bigg(\int_0^{T\wedge \mathfrak t_R}&\|\pi\|_{W^{2,2}_x}^2\dt\bigg)^{\frac{r}{4}}\bigg]\leq\,c\,\E\bigg[\bigg(\int_0^{T\wedge \mathfrak t_R}\|\bfu\cdot\nabla\bfu\|_{W^{1,2}_x}^2\dt\bigg)^{\frac{r}{4}}\bigg]\\
&\leq\,c\,\E\bigg[\bigg(\int_0^{T\wedge \mathfrak t_R}\|\nabla\bfu\|_{L^{4}_x}^4\dt+\int_0^{T\wedge \mathfrak t_R}\|\bfu\|_{L^q_x}^2\|\nabla^2\bfu\|_{L^{p}_x}^2\dt\bigg)^{\frac{r}{4}}\bigg]\\
&\leq\,c\,\E\bigg[\bigg(\int_0^{T\wedge \mathfrak t_R}\|\bfu\|_{W^{1+\beta,2}_x}^4\dt+\sup_{0\leq t\leq T\wedge \mathfrak t_R}\|\bfu\|_{W^{1,2}_x}^2\int_0^{T\wedge \mathfrak t_R}\|\bfu\|_{W^{2+\beta,2}_x}^2\dt\bigg)^{\frac{r}{4}}\bigg]\\
&\leq\,c\,\E\bigg[\bigg(\sup_{0\leq t\leq T\wedge \mathfrak t_R}\|\bfu\|_{W^{1+\beta,2}_x}^2+\int_0^{T\wedge \mathfrak t_R}\|\bfu\|_{W^{2+\beta,2}_x}^2\dt\bigg)^{\frac{r}{2}}\bigg]
\end{align*}
using the embeddings $W^{1+\beta,2}(\mt,\R^2)\hookrightarrow W^{1,4}(\mt,\R^2)$ and $W^{2+\beta,2}(\mt,\R^2)\hookrightarrow W^{2,p}(\mt,\R^2)$, which hold for an appropriate choice of $\beta\in(0,1)$.
Hence using Lemma \ref{lem:reg} (c) completes the proof.
\end{proof}

\begin{corollary}\label{cor:uholder}
\begin{enumerate}
\item[(a)] Let the assumptions of Lemma \ref{lem:reg} (b) be satisfied for some $r>2$. For all $\alpha<\frac{1}{2}$ we have
\begin{align}
\label{eq:holder}
\E\Big[\Big(\|\bfu(\cdot\wedge\mathfrak t_R)\|_{C^\alpha([0,T];L^{2}_x)}\Big)^{\frac{r}{2}}\Big]\leq cR^{3r}\,\E\Big[1+\|\bfu_0\|_{W^{1,2}_x}^r\Big].
\end{align}
\item[(b)] Let the assumptions of Lemma \ref{lem:reg} (c) be satisfied for some $r>2$. For all $\alpha<\frac{1}{2}$ we have
\begin{align}
\label{eq:holder}
\E\Big[\Big(\|\bfu(\cdot\wedge\mathfrak t_R)\|_{C^\alpha([0,T];W^{1,2}_x)}\Big)^{\frac{r}{2}}\Big]\leq cR^{5r}\,\E\Big[1+\|\bfu_0\|_{W^{2,2}_x}^r+\|\bfu_0\|_{W^{1,2}_x}^{2r}\Big].
\end{align}
\end{enumerate}
Here $c=c(r,T,\alpha)>0$ is independent of $R$.
\end{corollary}
\begin{proof}

As in \cite[Corollary 2.6]{BrDo} we can combine Lemmas \ref{lem:reg} and \ref{lem:pressure} to conclude the result concerning the time regularity of $\bfu$ form {\bf (a)}. As far is {\bf (b)} is concerned we analyse each term in equation \eqref{eq:pressure} separately. Lemma \ref{lem:reg} {\bf (b)} implies
\begin{align*}
\int_0^{\cdot\wedge\mathfrak t_R}\Delta\bfu\ds\in L^r(\Omega;L^{2}(0,T;W^{\beta,2}(\mt,\R^2))),
\end{align*}
whereas Lemmas \ref{lem:reg} {\bf (c)} and \ref{lem:pressure} {\bf (b)} yield
\begin{align*}
\int_0^{\cdot\wedge\mathfrak t_R}\big(\Div(\bfu\otimes\bfu)+\nabla\pi\big)\ds\in L^{\frac{r}{2}}(\Omega;L^{2}(0,T;W^{\beta,2}(\mt,\R^2))).
\end{align*}
Finally, we have
\begin{align*}
\int_0^{\cdot\wedge\mathfrak t_R}\Phi(\bfu)\,\dd W\in L^{r}(\Omega;C^{\alpha}([0,T];L^2(\mt,\R^2))).
\end{align*}
by combing Lemma \ref{lem:reg} {\bf (a)} with \eqref{eq:phi1a}. 
We conclude that
\begin{align*}
\E\Big[\Big(\|\bfu(\cdot\wedge\mathfrak t_R)\|_{C^\alpha([0,T];W^{\beta,2}_x)}\Big)^{\frac{r}{2}}\Big]\leq cR^{5r}\,\E\Big[1+\|\bfu_0\|_{W^{2,2}_x}^r+\|\bfu_0\|_{W^{1,2}_x}^{2r}\Big].
\end{align*}
for all $\beta<1$. Interpolating this with the estimate from Lemma \ref{lem:reg} {\bf (c)} gives the claim.
\end{proof}

\section{Error analysis: direct comparison}
\label{sec:error}
Now we consider a fully practical scheme combining a semi-implicit Euler scheme in time with a finite element approximation in space. It is defined on the given  filtered probability space $(\Omega,\mathfrak F,(\mathfrak F_t),\p)$ on which $W$ as well as the maximal strong pathwise solution to \eqref{eq:SNS} are defined.
For a given $h>0$ let $\bfu_{h,0}$ be an $\mathfrak F_0$-measurable random variable with values in $V^h_{\Div}(\mt,\R^2)$ (for instance $\Pi_h\bfu_0$; see \eqref{eq:stab}). We aim at constructing iteratively a sequence of random variables $(\bfu_{h,m},p_{h,m})$ such that
for every $(\bfphi, \chi) \in V^h(\mt,\R^2) \times P^h(\mt)$ it holds true $\p$-a.s.
\begin{align}\label{txdiscrpi}
\begin{aligned}
\int_{\mt}\bfu_{h,m}\cdot\bfvarphi \dx &+\tau\int_{\mt}\big((\bfu_{h,m-1}\cdot\nabla)\bfu_{h,m}+(\Div\bfu_{h,m-1})\bfu_{h,m}\big)\cdot\bfphi\dx\\
&\qquad+\mu\,\tau\int_{\mt}\nabla\bfu_{h,m}:\nabla\bfphi\dx-\tau\int_{\mt}p_{h,m}\,\Div\bfvarphi\dx\\&=\int_{\mt}\bfu_{h,m-1}\cdot\bfvarphi \dx+\int_{\mt}\Phi(\bfu_{h,m-1})\,\Delta_mW\cdot \bfvarphi\dx\, , \\
\int_{\mt} {\rm div} {\bf u}_{h,m} \cdot \chi&\dx = 0\, ,
\end{aligned}
\end{align}
where $\Delta_m W=W(t_m)-W(t_{m-1})$. Here the interval $[0,T]$ is decomposed into an equidistant grid of time points $t_m=m\tau=m\frac{T}{M}$ with $M\in\N$.
For our theoretical analysis it is convenient to work with the pressure-free formulation of \eqref{txdiscrpi}: For every $\bfphi\in V^h_{\Div}(\mt,\R^2)$ it holds true $\p$-a.s.
\begin{align}\label{txdiscr}
\begin{aligned}
\int_{\mt}&\bfu_{h,m}\cdot\bfvarphi \dx +\tau\int_{\mt}\big((\bfu_{h,m-1}\cdot\nabla)\bfu_{h,m}+(\Div\bfu_{h,m-1})\bfu_{h,m}\big)\cdot\bfphi\dx\\
&+\mu\,\tau\int_{\mt}\nabla\bfu_{h,m}:\nabla\bfphi\dx=\int_{\mt}\bfu_{h,m-1}\cdot\bfvarphi \dx+\int_{\mt}\Phi(\bfu_{h,m-1})\,\Delta_mW\cdot \bfvarphi\dx.
\end{aligned}
\end{align}
We quote the following result concerning the solution $(\bfu_{h,m})_{m=1}^M$ to \eqref{txdiscr} from \cite[Lemma 3.1]{BCP}.

\begin{lemma}\label{lemma:3.1BCP}
Fix $T>0$.
Assume that $\bfu_{h,0}\in L^{2^q}(\Omega,V_{\Div}^h(\mt,\R^2))$ with $q\in\N$ is an $\mathfrak F_0$-measurable random variable.
Suppose that $\Phi$ satisfies \eqref{eq:phi0}. Then the iterates $(\bfu_{h,m})_{m=1}^M$ given by \eqref{txdiscr}
are $(\mathfrak F_{t_m})$-measurable. Moreover, the following estimate holds uniformly in $M$ and $h$:
\begin{align}
\label{lem:4.1}\E\bigg[\max_{1\leq m\leq M}\|\bfu_{h,m}\|^{2^q}_{L^{2}_x}+\tau\sum_{k=1}^M\|\bfu_{h,m}\|^{2^{q}-2}_{L^{2}_x}\|\nabla\bfu_{h,m}\|^2_{L^2_x}\bigg]&\leq\,c(q,T)\E\Big[1+\|\bfu_{h,0}\|_{L_x^2}^{2^q}\Big].
\end{align}
\end{lemma}
Our error analysis for \eqref{txdiscr} is based on an auxiliary problem which coincides with \eqref{txdiscr} until a discrete stopping time. As we shall see below both problems coincide with high probability.
For every $m\geq 1$ we introduce
 the discrete stopping time
\begin{align}\mathfrak t_m^R:=\max_{1\leq n\leq m}\big\{t_n:t_n\leq \mathfrak t_R\big\},\label{eq:tmR}
\end{align}
which is obviously $(\mathfrak F_{t_m})$-stopping time (but not an $(\mathfrak F_t)$-stopping time). 
Setting
$\tau_{m}^R:=\mathfrak t_m^R-\mathfrak t_{m-1}^R$
we introduce $\bfu_{h,m}^R$ as the $V_{\Div}^h(\mt,\R^2)$-valued solution to
\begin{align}\label{txdiscrR}
\begin{aligned}
\int_{\mt}&\bfu_{h,m}^R\cdot\bfvarphi \dx +\tau_{m}^R\int_{\mt}\big((\bfu^R_{h,m-1}\cdot\nabla)\bfu_{h,m}^R+(\Div\bfu_{h,m-1}^R)\bfu_{h,m}^R\big)\cdot\bfphi\dx\\
&+\mu\,\tau_{m}^R\int_{\mt}\nabla\bfu_{m}^R:\nabla\bfphi\dx=\int_{\mt}\bfu_{h,m-1}^R\cdot\bfvarphi \dx+\frac{\tau_{m}^R}{\tau}\int_{\mt}\Phi(\bfu_{h,m-1}^R)\,\Delta_m W\cdot \bfvarphi\dx
\end{aligned}
\end{align}
for every $\bfphi\in V^h_{\Div}(\mt,\R^2)$. Obviously $\bfu_{h,m}^R=\bfu_{h,m}$ in $[t_m= \mathfrak t_m^R]$.
Our main effort is dedicated to the proof of an error estimate
for \eqref{txdiscrR} in the following theorem, for which $R>0$ is a fixed truncation parameter and $T>0$ an arbitrary but fixed time.

\begin{theorem}\label{thm:4}
Let $\bfu_0\in L^8(\Omega,W^{2,2}(\mt,\R^2))\cap L^{20}(\Omega;W^{1,2}_{0,\Div}(\mt,\R^2))$ be $\F_0$-measurable and assume that $\Phi$ satisfies \eqref{eq:phi0}--\eqref{eq:phi2b}. Let $$(\bfu,(\mathfrak{t}_R)_{R\in\N},\mathfrak{t})$$ be the unique maximal global strong solution to \eqref{eq:SNS} in the sense of Definition \ref{def:maxsol}.
Let $(\mathfrak t_m^R)_{m=1}^M$ be defined by \eqref{eq:tmR}.
Then we have for all $R\in\N$ and all $\alpha<\frac{1}{2}$, $\beta<1$
\begin{align}\label{eq:thm:4}
\begin{aligned}
\E\bigg[\max_{1\leq m\leq M}\|\bfu(\mathfrak t^R_{m})-\bfu_{h,m}^R\|_{L^2_x}^2&+\sum_{m=1}^M \tau_{m}^R\|\nabla\bfu(\mathfrak t_{m}^R)-\nabla\bfu_{h,m}^R\|_{L^2_x}^2\bigg)\bigg]\\&\leq \,ce^{cR^4}\,\big(h^{2\beta}+\tau^{2\alpha}\big),
\end{aligned}
\end{align}
where $(\bfu_{h,m}^R)_{m=1}^M$ is the solution to \eqref{txdiscrR} with $\bfu_{h,0}^R=\Pi_h\bfu_0$.
The constant $c$ in \eqref{eq:thm:4} is independent of $\tau$, $h$ and $R$.
\end{theorem}
\begin{remark}\label{rem}
{\bf 1.} In previous papers concerning the periodic problem, in particular \cite{CP}, the idea is to consider the equation for the error in the $m$-th step and multiply by the indicator function of a set $\Omega^{h,\tau}_{m-1}\subset\Omega$. Hereby $\Omega^{h,\tau}_{m-1}\subset\Omega$ is $\mathfrak F_{t_{m-1}}$-measurable and certain quantities up to time $t_{m-1}$ are bounded in $\Omega^{h,\tau}_{m-1}$. It is, however, not necessary to control the continuous solution in this way since global estimates are available, see, e.g., \cite[Lemma 2.1]{CP} or \cite[Lemma 2]{BrDo}.

In our situation, having only stopped estimates as in Lemma \ref{lem:reg}, it is necessary to also control the continuous solution. For certain quantities, having control until time $t_{m-1}$ is not sufficient  (see, for instance, the estimates for $I_2(m)$ and $I_3(m)$ below, where norms of $\bfu$ over $[t_{m-1},t_m]$ appear). Using $\mathfrak F_{t_{m}}$-measurable sets $\Omega^{h,\tau}_{m}\subset\Omega$ instead is not possible either as it destroys the martingale property
of $\mathscr M^1$ given below in \eqref{eq:2908}.

Both problems are overcome by the use of the discrete stopping time $\mathfrak t_m^R$: we can control norms of  $\bfu$ over $[t_{m-1},t_m]$, and $\mathscr M^1$ is estimated at time $\mathfrak t_R\geq \mathfrak t_m^R$ such that the martingale property can be used.

 {\bf 2.} The (discrete) gradient of the noise term in (\ref{txdiscrpi}) need not be subtracted here, as is in \cite{FePrVo}, since a {\em simultaneous  space-time} error analysis is used to prove Theorem \ref{thm:main} below.
\end{remark}

Our main result is now a direct consequence of Theorem \ref{thm:4}: Setting for $\varepsilon>0$ arbitrary $R=c^{-1/4}\sqrt[4]{-\varepsilon \log \min\{h,\tau\}}$, we have for any $\xi>0$
\begin{align*}
&\mathbb P\bigg[\max_{1\leq m\leq M}\|\bfu(t_m)-\bfu_{h,m}\|_{L^2_x}^2+\sum_{m=1}^M \tau\|\nabla\bfu(t_m)-\nabla\bfu_{h,m}\|_{L^2_x}^2>\xi\,\big(h^{2\beta-2\varepsilon}+\tau^{2\alpha-2\varepsilon}\big)\bigg]\\
&\leq \mathbb P\bigg[\max_{1\leq m\leq M}\|\bfu(\mathfrak t^R_{m})-\bfu_{h,m}^R\|_{L^2_x}^2+\sum_{m=1}^M \tau_{m}^R\|\nabla\bfu(\mathfrak t_{m}^R)-\nabla\bfu_{h,m}^R\|_{L^2_x}^2>\xi\,\big(h^{2\beta-2\varepsilon}+\tau^{2\alpha-2\varepsilon}\big)\bigg]\\
&\qquad+\mathbb P\big[\{\mathfrak t_R<T\}\big]
\rightarrow 0
\end{align*}
as $h,\tau\rightarrow0$ (recall that $\mathfrak t_R\rightarrow \infty$ $\p$-a.s. by Theorem \ref{thm:inc2d} and that $\mathfrak t_M^R<t_M$ implies $\mathfrak t_R<T$).
Relabeling $\alpha$ and $\beta$ we have proved the following result.

\begin{theorem}\label{thm:main}
Let $(\Omega,\mf,(\mf_t)_{t\geq0},\prst)$ be a given stochastic basis with a complete right-con\-ti\-nuous filtration and an $(\mf_t)$-cylindrical Wiener process $W$.
Let $\bfu_0\in L^8(\Omega,W^{2,2}(\mt,\R^2))\cap L^{20}(\Omega;W^{1,2}_{0,\Div}(\mt,\R^2))$ be $\F_0$-measurable and assume that $\Phi$ satisfies \eqref{eq:phi0}--\eqref{eq:phi2b}. Let $$(\bfu,(\mathfrak{t}_R)_{R\in\N},\mathfrak{t})$$ be the unique maximal global strong solution to \eqref{eq:SNS} from Theorem \ref{thm:inc2d}.
Then we have for any $\xi>0$, $\alpha<\frac{1}{2}$, $\beta<1$ 
\begin{align*}
&\mathbb P\bigg[\max_{1\leq m\leq M}\|\bfu(t_m)-\bfu_{h,m}\|_{L^2_x}^2+\sum_{m=1}^M \tau\|\nabla\bfu(t_m)-\nabla\bfu_{h,m}\|_{L^2_x}^2>\xi\,\big(h^{2\beta}+\tau^{2\alpha}\big)\bigg]\rightarrow 0
\end{align*}
as $h,\tau\rightarrow0$,
where $(\bfu_{h,m})_{m=1}^M$ is the solution to \eqref{txdiscr} with $\bfu_{h,0}=\Pi_h\bfu_0$.
\end{theorem}
In order to finish the proof of our main result stated in Theorem \ref{thm:main} above, we focus now on proving the error estimate from Theorem \ref{thm:4}
concerning the auxiliary problem \eqref{txdiscrR}.

\begin{proof}[Proof of Theorem \ref{thm:4}]
Define the error $\bfe_{h,m}=\bfu(\mathfrak t_m^R)-\bfu_{h,m}^R$. Subtracting \eqref{eq:pressure} and \eqref{txdiscrR} and recalling that functions from $W^{1,2}_0(\mt,\R^2)$ are admissible in \eqref{eq:pressure} we obtain 
\begin{align*}
\int_{\mt}&\bfe_{h,m}\cdot\bfvarphi \dx +\int_{\mathfrak t_{m-1}^R}^{\mathfrak t_m^R}\int_{\mt}\mu\nabla\bfu(\sigma):\nabla\bfphi\dx\ds-\tau_{m}^R\int_{\mt}\mu\nabla\bfu^R_{h,m}:\nabla\bfphi\dx\\&=\int_{\mt}\bfe_{h,m-1}\cdot\bfvarphi \dx-\int_{\mathfrak t_{m-1}^R}^{\mathfrak t_m^R}\int_{\mt}(\bfu(\sigma)\cdot\nabla)\bfu(\sigma)\cdot\bfphi\dxs\\&\qquad+\tau_{m}^R\int_{\mt}\big((\bfu_{h,m-1}^R\cdot\nabla)\bfu_{h,m}^R+(\Div\bfu^R_{h,m-1})\bfu_{h,m}^R\big)\cdot\bfphi\dx\\
&\qquad+\int_{\mt}\int_{\mathfrak t_{m-1}^R}^{\mathfrak t_m^R}\Phi(\bfu(\sigma))\,\dd W\cdot \bfvarphi\dx-\int_{\mt}\int_{\mathfrak t_{m-1}^R}^{\mathfrak t_m^R}\Phi(\bfu_{h,m-1}^R)\,\dd W\cdot \bfvarphi\dx\\
&\qquad+\int_{\mathfrak t_{m-1}^R}^{\mathfrak t_m^R}\int_{\mt}\pi(\sigma)\,\Div\bfphi\dx\ds
\end{align*}
for every $\bfphi\in V^h(\mt,\R^2)$, which is equivalent to
\begin{align*}
\begin{aligned}
\int_{\mt}&\bfe_{h,m}\cdot\bfvarphi \dx +\tau_{m}^R\int_{\mt}\mu\Big(\nabla\bfu(\mathfrak t_{m}^R)-\nabla\bfu^R_{h,m}\Big):\nabla\bfphi\dx\\&=\int_{\mt}\bfe_{h,m-1}\cdot\bfvarphi \dx+\int_{\mathfrak t_{m-1}^R}^{\mathfrak t_m^R}\int_{\mt}\mu\big(\nabla\bfu(\mathfrak t_{m}^R)-\nabla\bfu(\sigma)\big):\nabla\bfphi\dx\ds
\\
&\qquad+\int_{\mathfrak t_{m-1}^R}^{\mathfrak t_m^R}\int_{\mt}\Big((\bfu(\mathfrak t_{m-1}^R)\cdot\nabla)\bfu(\mathfrak t_{m}^R)-(\bfu(\sigma)\cdot\nabla)\bfu(\sigma)\Big)\cdot\bfphi\dxs\\
&\qquad-\tau_{m}^R\int_{\mt}\Big((\bfu(\mathfrak t_{m-1}^R)\cdot\nabla)\bfu(\mathfrak t_{m}^R)-\big((\bfu^R_{h,m-1}\cdot\nabla)\bfu^R_{h,m}+(\Div\bfu^R_{h,m-1})\bfu^R_{h,m}\big)\Big)\cdot\bfphi\dx\\
&\qquad+\int_{\mt}\int_{\mathfrak t_{m-1}^R}^{\mathfrak t_m^R}\big(\Phi(\bfu(\sigma))-\Phi(\bfu^R_{h,m-1})\big)\,\dd W\cdot \bfvarphi\dx+\int_{\mathfrak t_{m-1}^R}^{\mathfrak t_m^R}\int_{\mt}\pi(\sigma)\,\Div\bfphi\dxs.
\end{aligned}
\end{align*}
Setting $\bfphi=\Pi_h\bfe_{h,m}$ and applying the identity $\bfa\cdot(\bfa-\bfb)=\frac{1}{2}\big(|\bfa|^2-|\bfb|^2+|\bfa-\bfb|^2\big)$ (which holds for any $\bfa,\bfb\in\mathbb R^n$) we gain
\begin{align*}
&\frac{1}{2}\big(\|\Pi_h\bfe_{h,m}\|_{L^2_x}^2-\|\Pi_h\bfe_{h,m-1}\|_{L^2_x}^2+\|\Pi_h\bfe_{h,m}-\Pi_h\bfe_{h,m-1}\|_{L^2_x}^2\big)+\tau_{m}^R\mu\|\nabla\bfe_{h,m}\|^2_{L^2_x}\\
&=\tau_{m}^R\int_{\mt}\mu\nabla\bfe_{h,m}:\nabla\big(\bfu(\mathfrak t_{m}^R)-\Pi_h\bfu(\mathfrak t_{m}^R)\big)\dx\\
&+\int_{\mathfrak t_{m-1}^R}^{\mathfrak t_{m}^R}\int_{\mt}\mu\big(\nabla\bfu(\mathfrak t_{m}^R)-\nabla\bfu(\sigma)\big):\nabla\Pi_h\bfe_{h,m}\dx\ds
\\
&+\int_{\mathfrak t_{m-1}^R}^{\mathfrak t_{m}^R}\int_{\mt}\Big((\bfu(\mathfrak t_{m-1}^R\cdot\nabla)\bfu(\mathfrak t_{m}^R)-(\bfu(\sigma)\cdot\nabla)\bfu(\sigma)\Big)\cdot\Pi_h\bfe_{h,m}\dxs\\
&-\tau_{m}^R\int_{\mt}\Big((\bfu(\mathfrak t_{m-1}^R)\cdot\nabla)\bfu(\mathfrak t_{m}^R)-\big((\bfu_{h,m-1}^R\cdot\nabla)\bfu^R_{h,m}+(\Div\bfu^R_{h,m-1})\bfu_{h,m}^R\big)\Big)\cdot\Pi_h\bfe_{h,m}\dx\\
&+\int_{\mt}\int_{\mathfrak t_{m-1}^R}^{\mathfrak t^R_m}\big(\Phi(\bfu(\sigma))-\Phi(\bfu^R_{h,m-1})\big)\,\dd W\cdot \Pi_h\bfe_{h,m}\dx\\
&+\int_{\mathfrak t_{m-1}^R}^{\mathfrak t_{m}^R}\int_{\mt}\pi(\sigma)\,\Div\Pi_h\bfe_{h,m}\dxs\\
&=:I_1(m)+\dots +I_6(m).
\end{align*}
Eventually, we will take the maximum with respect to $m\in\{1,\dots,M\}$ and apply expectations. Let us explain how to deal with $\E\big[\max_m I_1(m)],\dots ,\E[\max_mI_6(m)]$ independently.

We clearly have for any $\kappa>0$
\begin{align*}
I_1(m)&\leq \,\kappa\tau_{m}^R\|\nabla\bfe_{h,m}\|_{L^2_x}^2+c(\kappa)\tau_{m}^R\|\nabla(\bfu(\mathfrak t_{m}^R)-\Pi_h\bfu(\mathfrak t_{m}^R)\|^2_{L^2_x}\\
&\leq \,\kappa\tau_{m}^R\|\nabla\bfe_{h,m}\|^2_{L^2_x}+c(\kappa)\tau h^{2\beta}\|\bfu(\mathfrak t_{m}^R)\|^2_{W^{1+\beta,2}_x}
\end{align*}
due to the $W^{1,2}_x$-stability of
$\Pi_h$, cf.~\eqref{eq:stab'}. Note that the expectation of the last term may be bounded with the help of Lemma \ref{lem:reg} (c)  
using $\mathfrak t_{m}^R\leq\mathfrak t_R$. We continue
with $I_2(m)$, for which we obtain
\begin{align*}
I_2(m)&\leq \,\kappa\tau_{m}^R\|\nabla\Pi_h\bfe_{h,m}\|_{L^2_x}^2+c(\kappa)\int_{\mathfrak t_{m-1}^R}^{\mathfrak t_{m}^R}\|\nabla(\bfu(\mathfrak t_{m}^R)-\bfu(\sigma))\|_{L^2_x}^2\ds\\
&\leq \,\kappa\tau_{m}^R\|\nabla\bfe_{h,m}\|^2_{L^2_x}+\,\kappa\tau_{m}^R\|\nabla(\bfu(\mathfrak t_m^R)-\Pi_h\bfu(\mathfrak t_m^R))\|_{L^2_x}^2+c(\kappa)\tau^{1+2\alpha}\|\nabla\bfu\|_{C^\alpha([\mathfrak t_{m-1}^R,\mathfrak t_{m}^R];L^2_x)}^2,
\end{align*}
where the last term can be controlled by Corollary \ref{cor:uholder} and the second last one by \eqref{eq:stab'} and \ref{lem:reg} (c) as for $I_2(m)$.
We proceed by
\begin{align*}
I_3(m)&=-\int_{\mathfrak t_{m-1}^R}^{\mathfrak t_{m}^R}\int_{\mt}\Big(\bfu(\mathfrak t_{m}^R)\otimes\bfu(\mathfrak t_{m-1}^R)-\bfu(\sigma)\otimes\bfu(\sigma)\Big):\nabla\Pi_h\bfe_{h,m}\dxs\\
&\leq \,\kappa\tau_{m}^R\|\nabla\Pi_h\bfe_{h,m}\|^2_{L^2_x}+c(\kappa)\int_{\mathfrak t_{m-1}^R}^{\mathfrak t_{m}^R}\|\bfu(\mathfrak t_{m}^R)\otimes\bfu(\mathfrak t_{m-1}^R)-\bfu(\sigma)\otimes\bfu(\sigma)\|^2_{L^2_x}\ds\\
&\leq \,\kappa\tau_{m}^R\|\nabla\bfe_{h,m}\|^2_{L^2_x}+\,\kappa\tau_{m}^R\|\nabla(\bfu(\mathfrak t_m^R)-\Pi_h\bfu(\mathfrak t_m^R))\|_{L^2_x}\\&\qquad+c(\kappa)\tau^{1+2\alpha}\|\bfu\|^2_{L^\infty((\mathfrak t_{m-1}^R,\mathfrak t_{m}^R)\times\mt)}\|\bfu\|_{{C^\alpha([\mathfrak t_{m-1}^R,\mathfrak t_{m}^R];L^2_x)}}^2\\
&\leq \,\kappa\tau_{m}^R\|\nabla\bfe_{h,m}\|^2_{L^2_x}+c(\kappa)\tau h^{2\beta}\|\bfu(\mathfrak t_{m}^R)\|^2_{W^{1+\beta,2}_x}\\&\qquad+c(\kappa)\tau^{1+2\alpha}\|\bfu\|^2_{L^\infty(\mathfrak t_{m-1}^R,\mathfrak t_{m}^R;W^{2,2}_x)}\|\bfu\|_{{C^\alpha([\mathfrak t_{m-1}^R,\mathfrak t_{m}^R];L^2_x)}}^2,
\end{align*}
where we used Sobolev's embedding $W^{2,2}(\mt,\R^2)\hookrightarrow L^\infty(\mt,\R^2)$. We can apply again
Lemma \ref{lem:reg} (c)  and Corollary \ref{cor:uholder} to the last term.
The term $I_4(m)$ can be decomposed as
\begin{align*}
I_4(m)&=I_4^1(m)+I_4^2(m)+I_4^3(m),\\
I_4^1(m)&=-\tau_{m}^R\int_{\mt}(\bfu(\mathfrak t_{m-1}^R\cdot\nabla)\bfe_{h,m}\cdot\big(\bfu(\mathfrak t_{m}^R)-\Pi_h\bfu(\mathfrak t_{m}^R)\big)\dx,\\
I_4^2(m)&=\tau_{m}^R\int_{\mt}(\bfe_{h,m-1}\cdot\nabla)\bfe_{h,m}\cdot\big(\bfu(\mathfrak t_{m}^R)-\Pi_h\bfu(\mathfrak t_{m}^R)\big)\dx\\
&\qquad+\tau_{m}^R\int_{\mt}(\Div\bfe_{h,m-1})\bfe_{h,m}\cdot\big(\bfu(\mathfrak t_{m}^R)-\Pi_h\bfu(\mathfrak t_{m}^R)\big)\dx,\\
I_4^3(m)&=-\tau_{m}^R\int_{\mt}(\bfe_{h,m-1}\cdot\nabla)\Pi_h\bfe_{h,m}\cdot\bfu(\mathfrak t_{m}^R)\dx\\
&\qquad-\tau_{m}^R\int_{\mt}(\Div\bfe_{h,m-1})\Pi_h\bfe_{h,m}\cdot\bfu(\mathfrak t_{m}^R)\dx.
\end{align*}
We obtain for any $\kappa>0$
\begin{align*}
I_4^1(m)&\leq\,\tau_{m}^R\|\nabla\bfe_{h,m}\|^2_{L^2_x}\|\bfu(\mathfrak t_{m-1}^R)\|_{L^4_x}\|\bfu(\mathfrak t_{m}^R)-\Pi_h\bfu(\mathfrak t_{m}^R)\|_{L^4_x}\\
&\leq\,c\tau_{m}^R\|\nabla\bfe_{h,m}\|^2_{L^2_x}\|\bfu(\mathfrak t_{m-1}^R)\|_{W^{1,2}_x}\|\bfu(\mathfrak t_{m}^R)-\Pi_h\bfu(\mathfrak t_{m}^R)\|_{L^2_x}^{\frac{1}{2}}\|\nabla(\bfu(\mathfrak t_{m}^R)-\Pi_h\bfu(\mathfrak t_{m}^R))\|_{L^2_x}^{\frac{1}{2}}\\
&\leq \,c\tau_{m}^R\|\nabla\bfe_{h,m}\|_{L^2_x}h^{1+\beta/2} R\|\bfu(\mathfrak t_{m}^R)\|_{W^{1+\beta,2}_x}\\
&\leq \,\kappa\tau_{m}^R\|\nabla\bfe_{h,m}\|^2_{L^2_x}+c(\kappa) h^{2+\beta} R^2\tau_{m}^R\|\bfu(\mathfrak t_{m}^R)\|_{W^{1+\beta,2}_x}^2
\end{align*}
by the embedding $W^{1,2}(\mt,\R^2)\hookrightarrow L^4(\mt,\R^2)$, Ladyshenskaya's inequality, the definition of $\mathfrak t_m^R$, and \eqref{eq:stab'}. The first term can be absorbed for $\kappa$ small enough, whereas the second one (in summed form and expectation) is bounded by $h^{2+\beta}R^{12}$ due Lemma \ref{lem:reg} (c). Similarly, we have
\begin{align*}
I_4^2(m)&\leq \tau_{m}^R\|\nabla\bfe_{h,m}\|_{L^2_x}\|\bfe_{h,m-1}\|_{L^4_x}\|\bfu(\mathfrak t_{m}^R)-\Pi_h\bfu(\mathfrak t_{m}^R)\|_{L^4_x}\\
&\qquad+\tau_{m}^R\|\nabla\bfe_{h,m-1}\|_{L^2_x}\|\bfe_{h,m}\|_{L^4_x}\|\bfu(\mathfrak t_{m}^R)-\Pi_h\bfu(\mathfrak t_{m}^R)\|_{L^4_x}\\
&\leq \,c\tau_{m}^Rh^{1+\beta/2}\|\nabla\bfe_{h,m}\|_{L^2_x}\|\bfe_{h,m-1}\|_{L^2_x}^{\frac{1}{2}}\|\nabla\bfe_{h,m-1}\|^{\frac{1}{2}}_{L^2_x}\|\bfu(\mathfrak t_{m}^R)\|_{W^{1+\beta,2}_x}\\
&\qquad+c\tau_{m}^Rh^{1+\beta/2}\|\nabla\bfe_{h,m-1}\|_{L^2_x}\|\bfe_{h,m}\|_{L^2_x}^{\frac{1}{2}}\|\nabla\bfe_{h,m}\|_{L^2_x}^{\frac{1}{2}}\|\bfu(\mathfrak t_{m}^R)\|_{W^{1+\beta,2}_x}\\
&\leq\,\kappa\tau_{m}^R\Big(\|\nabla\bfe_{h,m-1}\|^2_{L^2_x}+\|\nabla\bfe_{h,m}\|^2_{L^2_x}\Big)\\
&\qquad+c(\kappa)\,\tau_{m}^R \,h^{4+2\beta}\Big(\max_{1\leq n\leq m}\|\bfe_{h,n}\|_{L^2_x}^2\Big)\|\bfu(\mathfrak t_{m}^R)\|^4_{W^{1+\beta}_x}.
\end{align*}
The last term (in summed form and expectation) can be controlled by Lemma \ref{lem:reg} (c) (with $r=8$) and Lemma \ref{lemma:3.1BCP} (with $q=2$). Note that we have either have $\bfu_{h,m}=\bfu_{h,m}^R$ or $\tau_m^R=0$.
Finally, by definition of $\mathfrak t_m^R$,
\begin{align*}
I_4^3(m)&\leq \tau_{m}^R\|\nabla\Pi_h\bfe_{h,m}\|_{L^2_x}\|\bfe_{h,m-1}\|_{L^4_x}\|\bfu(\mathfrak t_{m}^R)\|_{L^4_x}\\
&\qquad+ \tau_{m}^R\|\nabla\bfe_{h,m-1}\|_{L^2_x}\|\Pi_h\bfe_{h,m}\|_{L^4_x}\|\bfu(\mathfrak t_{m}^R)\|_{L^4_x}\\
&\leq \tau_{m}^R\|\nabla\bfe_{h,m}\|_{L^2_x}\|\bfe_{h,m-1}\|_{L^2_x}^{\frac{1}{2}}\|\nabla\bfe_{h,m-1}\|^{\frac{1}{2}}_{L^2_x}\|\bfu(\mathfrak t_{m}^R)\|_{W^{1,2}_x}\\
&\qquad+ \tau_{m}^R\|\nabla\bfe_{h,m-1}\|_{L^2_x}\|\Pi_h\bfe_{h,m}\|_{L^2_x}^{\frac{1}{2}}\|\nabla\Pi_h\bfe_{h,m}\|^{\frac{1}{2}}_{L^2_x}\|\bfu(\mathfrak t_{m}^R)\|_{W^{1,2}_x}\\
&\leq\,\kappa\tau_{m}^R\Big(\|\nabla\bfe_{h,m-1}\|^2_{L^2_x}+\|\nabla\Pi_h\bfe_{h,m}\|^2_{L^2_x}\Big)+c_\kappa\,\tau_{m}^R R^4\big(\|\Pi_h\bfe_{h,m}\|^2_{L^2_x}+\|\bfe_{h,m-1}\|^2_{L^2_x}\big)\\
&\leq\,\kappa\tau_{m}^R\Big(\|\nabla\bfe_{h,m-1}\|^2_{L^2_x}+\|\nabla\bfe_{h,m}\|^2_{L^2_x}\Big)+c(\kappa)\,\tau_{m}^R R^4\big(\|\Pi_h\bfe_{h,m}\|^2_{L^2_x}+\|\Pi_h\bfe_{h,m-1}\|^2_{L^2_x}\big)\\
&\qquad c(\kappa)\,\tau_m^R\|\nabla(\bfu(\mathfrak t_{m}^R)-\Pi_h\bfu(\mathfrak t_m^R))\|^2_{L^{2}_x}+c(\kappa)\,\tau_{m}^R R^4\|\bfu(\mathfrak t_{m-1}^R)-\Pi_h\bfu(\mathfrak t_{m-1}^R)\|^2_{L^2_x}.
\end{align*}
The second last term will be dealt with by Gronwall's lemma leading to a constant of the form $c e^{cR^4}$. The final line is bounded by 
$c(\kappa)\,\tau_{m}^R R^4h^{2\beta}\|\bfu(\mathfrak t_{m-1}^R)\|^2_{W^{1+\beta,2}_x}$ using \eqref{eq:stab'} and hence can be controlled by Lemma \ref{lem:reg} (c).

In order to estimate the stochastic term we write
\begin{align}
\mathscr M_{m}&=\sum_{n=1}^mI_5(n)=
\sum_{n=1}^m\int_{\mathcal{O}}\int_{\mathfrak t_{n-1}^R}^{\mathfrak t_{n}^R}\big(\Phi(\bfu)-\Phi(\bfu_{h,n-1}^R)\big)\,\dd W\cdot \Pi_h\bfe_{h,n}\dx\nonumber\\
&= \sum_{n=1}^m\int_{\mathcal{O}}\int_{\mathfrak t_{n-1}^R}^{\mathfrak t_{n}^R}\big(\Phi(\bfu)-\Phi(\bfu_{h,n-1}^R)\big)\,\dd W\cdot \Pi_h\bfe_{h,n-1}\dx\nonumber\\
&+ \sum_{n=1}^m\int_{\mathcal{O}}\int_{\mathfrak t_{n-1}^R}^{\mathfrak t_{n}^R}\big(\Phi(\bfu)-\Phi(\bfu_{h,n-1}^R)\big)\,\dd W\cdot \Pi_h(\bfe_{h,n}-\bfe_{h,n-1})\dx\nonumber\\
&= \int_{0}^{\mathfrak t_{m}^R}\sum_{n=1}^M\mathbf1_{[t_{n-1},t_n)}\int_{\mathcal{O}}\big(\Phi(\bfu)-\Phi(\bfu_{h,n-1}^R)\big)\,\dd W\cdot \Pi_h\bfe_{h,n-1}\dx\nonumber\\
&+ \sum_{n=1}^m\int_{\mathcal{O}}\int_{\mathfrak t_{n-1}^R}^{\mathfrak t_{n}^R}\big(\Phi(\bfu)-\Phi(\bfu_{h,n-1}^R)\big)\,\dd W\cdot \Pi_h(\bfe_{h,n}-\bfe_{h,n-1})\dx\nonumber\\
&=:\mathscr M^1(\mathfrak t_{m}^R)+\mathscr M_{m}^2.\label{eq:2908}
\end{align}
Since the process $(\mathscr M^1(t\wedge\mathfrak t_R))_{t\geq0}$ is an $(\mathfrak F_t)$-martingale we gain by the Burgholder-Davis-Gundy inequality (using that $\mathfrak t_M^R\leq \mathfrak t_R$ by definition)
\begin{align*}
&\E\bigg[\max_{1\leq m\leq M}\big|\mathscr M^1(\mathfrak t_m^R)\big|\bigg]\leq \E\bigg[\sup_{s\in[0,\mathfrak t_M^R]}\big|\mathscr M^1(s)\big|\bigg]\leq \E\bigg[\sup_{s\in[0,T]}\big|\mathscr M^1(s\wedge \mathfrak t_R)\big|\bigg]\\
&\leq\,c\,\E\bigg[\bigg(\int_{0}^{T \wedge\mathfrak t_{R}}\sum_{n=1}^M\mathbf1_{[t_{n-1},t_n)}\|\Phi(\bfu)-\Phi(\bfu_{h,n-1}^R)\|^2_{L_2(\mathfrak U,L^2_x)}\|\Pi_h\bfe_{h,n-1}\|^2_{L^2_x}\dt\bigg)^{\frac{1}{2}}\bigg]\\
&\leq\,c\,\E\bigg[\max_{1\leq n\leq M}\|\Pi_h\bfe_{h,n}\|_{L^2_x}\bigg(\int_{0}^{T\wedge\mathfrak t_R}\sum_{n=1}^M\mathbf1_{[t_{n-1},t_n)}\|\Phi(\bfu)-\Phi(\bfu_{h,n-1}^R)\|^2_{L_2(\mathfrak U,L^2_x)}\dt\bigg)^{\frac{1}{2}}\bigg]\\
&\leq\,\kappa\,\E\bigg[\max_{1\leq n\leq M}\|\Pi_h\bfe_{h,n}\|^2_{L^2_x}\bigg]+\,c_\kappa\,\E\bigg[\int_{0}^{T\wedge\mathfrak t_R}\sum_{n=1}^M\mathbf1_{[t_{n-1},t_n)}\|\bfu-\bfu_{h,n-1}^R\|_{L^2_x}^2\dt\bigg].
\end{align*}
Here, we also used \eqref{eq:phi0} as well as Young's inequality for $\kappa>0$ arbitrary.
Since $\bfu_{h,n-1}^R=\bfe_{h,n-1}+\bfu(\mathfrak t_{n-1}^R)$ is $V_{\Div}^h(\mt,\R^2)$-valued, we further estimate
\begin{align*}
&\E\bigg[\max_{1\leq m\leq M}\big|\mathscr M^1(\mathfrak t_m^R)\big|\bigg]\\&\leq\,\kappa\,\E\bigg[\max_{1\leq n\leq M}\|\Pi_h\bfe_{h,n}\|^2_{L^2_x}\bigg]+\,c(\kappa)\,\E\bigg[\int_{0}^{T\wedge\mathfrak t_R}\sum_{n=1}^M\mathbf1_{[t_{n-1},t_n)}\|\bfu-\bfu(\mathfrak t_{n-1}^R)\|_{L^2_x}^2\dt\bigg]\\
&\qquad+\,c(\kappa)\,\E\bigg[\int_{ 0}^{T\wedge\mathfrak t_R}\sum_{n=1}^M\mathbf1_{[t_{n-1},t_n)}\|\bfu(\mathfrak t_{n-1}^R)-\Pi_h\bfu(\mathfrak t_{n-1}^R)\|_{L^2_x}^2\dt\bigg]\\&\qquad+\,c(\kappa)\,\E\bigg[\int_{0}^{T\wedge\mathfrak t_R}\sum_{n=1}^M\mathbf1_{[t_{n-1},t_n)}\|\Pi_h\bfe_{h,n-1}\|_{L^2_x}^2\dt\bigg]
\end{align*}
We bound the last term by
\begin{align*}
\E\bigg[\int_{0}^{T\wedge\mathfrak t_R}\sum_{n=1}^M\mathbf1_{[t_{n-1},t_n)}\|\Pi_h\bfe_{h,n-1}\|_{L^2_x}^2\dt\bigg]&\leq \E\bigg[\sum_{n=1}^{M+1}\tau_n^R\|\Pi_h\bfe_{h,n-1}\|_{L^2_x}^2\dt\bigg]\\
&\leq \E\bigg[\sum_{n=0}^{M}\tau_n^R\|\Pi_h\bfe_{h,n}\|_{L^2_x}^2\dt\bigg]
\end{align*}
using that $\mathfrak t_R\wedge t_M\leq \mathfrak t^R_{M+1}$ and $\tau_n^R\leq\tau_{n-1}^R$ with $\tau_0^R:=\tau$.
 Applying \eqref{eq:stab'} as well as Lemma \ref{lem:reg} (b) and Corollary \ref{cor:uholder} (b) we gain
\begin{align*}
\E\bigg[\max_{1\leq m\leq M}\big|\mathscr M^1(\mathfrak t_m^R)\big|\bigg]&\leq\,\kappa\,\E\bigg[\max_{1\leq n\leq M}\|\Pi_h\bfe_{h,n}\|^2_{L^2_x}\bigg]+\,c(\kappa)\,\E\bigg[\sum_{n=0}^M \tau_{n}^R\|\Pi_h\bfe_{h,n}\|_{L^2_x}^2\bigg]\\
&+c(\kappa)\tau^{2\alpha}\E\big[\|\bfu\|_{C^\alpha([0,T\wedge\mathfrak t_R],L^2_x)}^2\big]+c(\kappa) h^{1+\beta}\E\bigg[\sup_{t\in[0,T]}\int_{\mathcal{O}}|\nabla\bfu(t\wedge \mathfrak t_R)|^2\dx\bigg]\\
&\leq\,\kappa\,\E\bigg[\max_{1\leq n\leq M}\|\Pi_h\bfe_{h,n}\|^2_{L^2_x}\bigg]+\,c(\kappa)\,\E\bigg[\sum_{n=0}^M \tau\|\Pi_h\bfe_{h,n}\|_{L^2_x}^2\bigg]\\&+c(\kappa)\tau^{2\alpha}R^{20}+c(\kappa) h^{1+\beta}R^6.
\end{align*}
Similarly: on using Cauchy-Schwartz inequality, Young's inequality, It\^{o}-isometry and \eqref{eq:phi0} we have for $\kappa>0$
\begin{align*}
&\E\bigg[\max_{1\leq m\leq M}|\mathscr M_{m}^2|\bigg]\\&\leq \E\bigg[ \sum_{n=1}^M\bigg( \kappa \|\Pi_h(\bfe_{h,n}-\bfe_{h,n-1})\|_{L^2_x}^2 +c(\kappa) \left\| \int_{\mathfrak t_{n-1}^R}^{\mathfrak t_{n}^R}\big(\Phi(\bfu)-\Phi(\bfu_{h,n-1}^R)\big)\,\dd W  \right\|_{L^2_x}^2\bigg) \bigg]\\
&\leq \kappa\E\bigg[ \sum_{n=1}^M \|\Pi_h(\bfe_{h,n}-\bfe_{h,n-1})\|_{L^2_x}^2 \bigg] + c(\kappa)\E\bigg[\sum_{n=1}^M\int_{ \mathfrak t_{n-1}^R}^{\mathfrak t_{n}^R}\|\bfu-\bfu_{h,n-1}^R\|_{L^2_x}^2\dt\bigg]\\
&\leq \kappa\E\bigg[ \sum_{n=1}^M \|\Pi_h(\bfe_{h,n}-\bfe_{h,n-1})\|_{L^2_x}^2 \bigg] +c(\kappa) \,\E\bigg[\sum_{n=1}^M \int_{ \mathfrak t_{n-1}^R}^{\mathfrak t_{n}^R}\|\bfu-\bfu(t_{n-1}^R)\|_{L^2_x}^2\dt\bigg]\\&\qquad+c_\kappa \,\E\bigg[\sum_{n=1}^M \tau_{n,h}^R\|\bfu(\mathfrak t_{n-1}^R)-\Pi_h\bfu(\mathfrak t_{n-1}^R)\|_{L^2_x}^2\bigg]+\,c(\kappa)\,\E\bigg[\sum_{n=1}^M \tau_{n}^R\|\Pi_h\bfe_{h,n-1}\|_{L^2_x}^2\bigg]\\
&\leq \kappa\E\bigg[ \sum_{n=1}^M \|\Pi_h(\bfe_{h,n}-\bfe_{h,n-1})\|_{L^2_x}^2 \bigg] +\,c(\kappa)\,\E\bigg[\sum_{n=1}^M \tau_{n}^R\|\Pi_h\bfe_{h,n-1}\|_{L^2_x}^2\bigg]\\&+c(\kappa)\tau^{2\alpha}R^{20}+c(\kappa)h^2R^6
\end{align*}
as a consequence Lemma \ref{lem:reg} (b) (using also \eqref{eq:stab}) and Corollary \ref{cor:uholder} (b).\\
Finally, we have by \eqref{eq:stabpi}
\begin{align*}
I_6(m)&=\int_{\mathfrak t_{m-1}^R}^{\mathfrak t_m^R}\int_{\mt}\big(\pi-\Pi_h^\pi\pi\big)\Div\Pi_h\bfe_{h,m}\dx\ds\\
&\leq \,c(\kappa)\int_{\mathfrak t_{m-1}^R}^{\mathfrak t_m^R} \,\|\pi-\Pi_h^\pi\pi\|_{L^2_x}^2\ds+\,\kappa\tau\,\|\nabla\Pi_h\bfe_{h,m}\|_{L^2_x}^2\\
&\leq \, c(\kappa) h^2\int_{\mathfrak t_{m-1}^R}^{\mathfrak t_m^R}\,\|\nabla\pi\|_{L^2_x}^2\ds+\,\kappa\tau\,\|\nabla\bfe_{h,m}\|_{L^2_x}^2,
\end{align*}
where $\kappa>0$ is arbitrary. The first term is summable in expectation with bound $c(\kappa) h^2 R^{12}$ due to Lemma \ref{lem:pressure} (b) and the last one can be absorbed. Collecting all estimates, choosing $\kappa$ small enough and applying Gronwall's lemma yields the claim.
\end{proof}

\end{document}